\documentclass[11pt, reqno]{amsart}
\usepackage{verbatim}
\usepackage{graphicx, epsfig, psfrag}
\usepackage{amsfonts,amsmath,amssymb,amsbsy,amsthm}
\usepackage{bm}
\usepackage{color}
\usepackage{float}
\usepackage{hyperref}
\usepackage{mathrsfs}
\usepackage{subfigure}
\usepackage{bbm}
\usepackage{xcolor}

\newtheorem{assumption}{Assumption}

\numberwithin{equation}{section}

\usepackage{environments}

\usepackage[top=2.5cm,bottom=2.5cm,left=2.5cm,right=2.5cm]{geometry}

\renewcommand{\paragraph}[1]{\subsubsection{#1}}

\def\R{\mathbb{R}}

\def\Rc{\mathcal{R}}

\def\<{\langle}
\def\>{\rangle}

\def\argmin{{\rm argmin}}

\def\mA{{\sf A}}
\def\mB{{\sf B}}
\def\mF{{\sf F}}

\def\a{{\rm a}}

\newcommand{\triple}{(\rho \alpha\beta)}
\newcommand{\tripleR}{(\rho \alpha\beta)\in \mathcal{R}}
\newcommand{\tripleTau}{(\tau\gamma\delta)}

\newcommand{\tripleIota}{(\tau\iota\chi)}

\definecolor{docol}{rgb}{0, 0.4, 0}
\definecolor{cocol}{rgb}{0.7, 0, 0}
\definecolor{ascol}{rgb}{0, 0, 0.7}

\definecolor{yscol}{HTML}{6622AA}

\def\XXint#1#2#3{{\setbox0=\hbox{$#1{#2#3}{\int}$ }
\vcenter{\hbox{$#2#3$ }}\kern-.6\wd0}}

\title[Dislocations in Multilattices]{Elastic Far-Field Decay from Dislocations in Multilattices}

\author{Derek Olson}
\address{Derek Olson\\
Third Wave Systems\\
6475 City W Pkwy\\
Eden Prairie, MN 55344\\
United States of America 
}
\email{derek.olson@thirdwavesys.com}

\author{Chirstoph Ortner}
\address{Christoph Ortner\\
Department of Mathematics\\
University of British Columbia\\
1984 Mathematics Road\\
Vancouver, British Columbia\\
Canada
}
\email{ortner@math.ubc.ca}

\author{Yangshuai Wang}
\address{Yangshuai Wang\\
Department of Mathematics\\
University of British Columbia\\
1984 Mathematics Road\\
Vancouver, British Columbia\\
Canada
}
\email{yswang2021@math.ubc.ca}

\author{Lei Zhang}
\address{Lei Zhang\\
Institute of Natural Sciences\\
Shanghai Jiao Tong University\\
800 Dongchuan Road\\
Shanghai, 200240\\
China}
\email{lzhang2012@sjtu.edu.cn}

\date{\today}

\thanks{}

\begin{document}

\begin{abstract}
   We precisely and rigorously characterise the decay of elastic fields generated by dislocations in crystalline materials, focusing specifically on the role of multilattices. Concretely, we establish that the elastic field generated by a dislocation in a multilattice can be decomposed into a continuum field predicted by a linearised Cauchy-Born elasticity theory, and a discrete and nonlinear core corrector representing the defect core. We demonstrate both analytically and numerically the consequences of this result for cell size effects in numerical simulations.
\end{abstract}

\maketitle

\section{Introduction}
A key approximation in all numerical simulations of crystalline defects is the boundary condition emulating the crystalline far-field. The ``quality'' of this boundary condition has a significant consequence for the severity of cell-size effects in such simulations. The study of these cell-size effects~\cite{Ehrlacher2013} and the development of new boundary conditions~\cite{2017-bcscrew} begins with a characterisation of the elastic field surrounding the defect core, which was initiated in \cite{Ehrlacher2013}. Such a characterisation is also interesting in other contexts, e.g., in the study of defect interactions~\cite{2014-dislift}. 

In the present work, we extend results concerning the decay of elastic far-fields generated by defects in crystalline materials to {\em dislocations in multilattices} (also referred to as complex crystals). The extension from simple lattices to multilattices is vital for physical applications as it enlarges the pool of admissible materials to include far more materials of interest to materials science. While results of this kind have been known in the materials science community from computational experiments and justified by associated continuum results dating back to Volterra~\cite{volterra}, we fill a gap in the existing literature by producing a rigorous proof of such decay estimates for dislocations in multilattices modeled via classical empirical potentials, i.e., in a discrete and fully nonlinear setting.

The decay estimates for the strain and strain gradients generated by a dislocation in the multilattice setting match those of the simple lattice, which is initially surprising for the following reason:  A fundamental tool that is used to prove decay of the elastic fields in the simple lattice setting is algebraic decay of the residual atomistic forces evaluated at a continuum elasticity predictor displacement.  These residual forces are shown to decay at a rate of $r^{-3}$ in~\cite{Ehrlacher2013} where $r$ is the distance from the dislocation core, but due to reduced symmetry these same forces do not in general decay as $r^{-3}$ in the multilattice setting.  However, we find that in the multilattice setting the {\em net force on a multilattice site} (the sum of the forces on all species of atoms at a single multilattice site) still decays as $r^{-3}$. This observation, along with a strong localisation of multilattice shifts then turns out to be sufficient to prove the expected decay of the strain fields for multilattice dislocations.

Our results can be used to establish convergence of numerical methods for simulating crystal defects including direct atomistic simulations~\cite{Ehrlacher2013} as well as various classes of multiscale methods~\cite{multiBlendingArxiv, buze2021numerical, 2021-qmmm3, chen17}. Indeed, such direct atomistic computational methods date back decades to~\cite{chang1966edge} which investigated computational methods for dislocations in iron (a simple crystal) and later works including~\cite{nandedkar1987atomic, Nandedkar1990} which investigated dislocations in silicon and other diamond cubic lattices (a multilattice). 


%

\subsection{Outline}

This paper is organised as follows. We begin our works in Section \ref{sec:notation} by giving a brief overview of the notation used in the paper. We then describe the multilattice structure and the specific material models admitted in our analysis of dislocations in Section \ref{sec:model}. Our main results, the potential energy for straight-line dislocations over an appropriate admissible space of multilattice displacements is well-defined (Theorem \ref{energy_thm}) and the decay of elastic fields (Theorem \ref{decay_thm}), are also presented here. In Section \ref{sec:numerics}, we end with a straightforward consequence of Theorem~\ref{decay_thm}: convergence of a direct atomistic method for a lattice statics simulation of a dislocation in an infinite crystal, followed by two numerical examples of edge dislocation confirming our theoretical predictions. Finally, some concluding remarks are given.  For simplicity of the presentations, we put the detailed proofs and analysis in Section \ref{sec:proof:energy_thm} and Section \ref{sec:proof:decay_thm}.

%
\subsection{Notation}
\label{sec:notation}
The symbol $C$ (or $c$) denotes a generic positive constant that may change from one line
of an estimate to the next. When estimating rates of decay or convergence, $C$
will always remain independent of approximation parameters such as the system size, the configuration of the lattice and the test functions. The dependence of $C$ will be clear from the context or stated explicitly. To further simplify notation we will often write $\lesssim$ to mean $\leq C$.
We use the symbol $\langle\cdot,\cdot\rangle$ to denote an abstract duality
pair between a Banach space and its dual space. The symbol $|\cdot|$ normally
denotes the Euclidean or Frobenius norm, while $\|\cdot\|$ denotes an operator
norm.
For a functional $E \in C^2(X)$, the first and second variations are denoted by
$\<\delta E(u), v\>$ and $\<\delta^2 E(u) v, w\>$ for $u,v,w\in X$, respectively.
The closed ball with radius $r$ and center $x$ is denoted by $B_r(x)$.

\section{Model Problem and Main Results}
\label{sec:model}
\subsection{Kinematics} \label{sec:kinematics}
We begin by letting $\mathcal{L}$ denote a Bravais lattice:
\[
\mathcal{L} = \mB \mathbb{Z}^3, \qquad \mB \in \mathbb{R}^{3 \times 3}_+.
\]
A multilattice, or complex crystal, is obtained by taking a union of shifted lattices,
\[
\mathcal{M} = \bigcup_{\alpha = 0}^{S-1} \mathcal{L} + p^{\rm ref}_\alpha,
\]
where each $p^{\rm ref}_\alpha \in \mathbb{R}^3$ represents a shift vector. The set of atoms in a unit cell are indexed by $\mathcal{S}:=\{0, 1, ..., S-1\}$.
Without loss of generality, we always assume $p^{\rm ref}_0 = 0$.  We refer to $\ell \in \mathcal{L}$ as a \textit{lattice site}, while $\ell+p^{\rm ref}_\alpha$ represents an actual atom location.
$\mathcal{L} + p^{\rm ref}_\alpha$ are the lattice locations for species $\alpha$.
Three common examples ($S=2$) of multilattices are shown in Figure \ref{fig:multi}.
\begin{figure}[!htb]
	\centering 
	\subfigure[B1 (NaCl type) structure]{
		\label{fig:b1}
		\includegraphics[height=3.5cm]{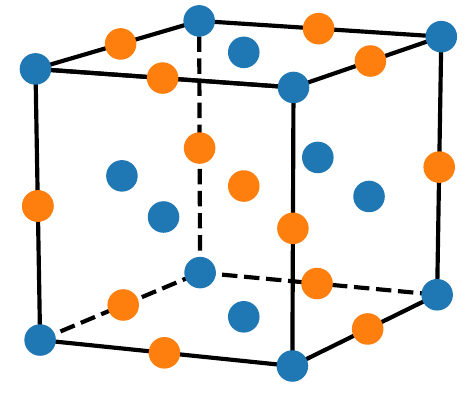}}
	\hspace{0.25cm} 
	\subfigure[B2 (CsCl type) structure]{
		\label{fig:b2} 
		\includegraphics[height=3.5cm]{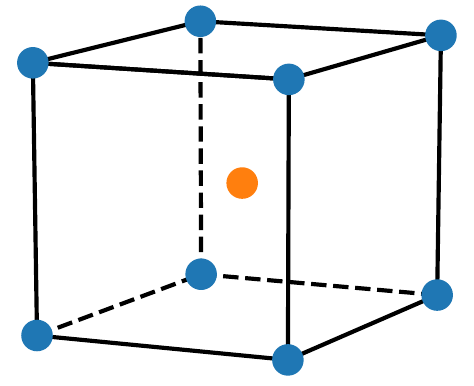}}
	\hspace{0.25cm} 
	\subfigure[B3 (ZnS type) structure]{
		\label{fig:b3}	\includegraphics[height=3.5cm]{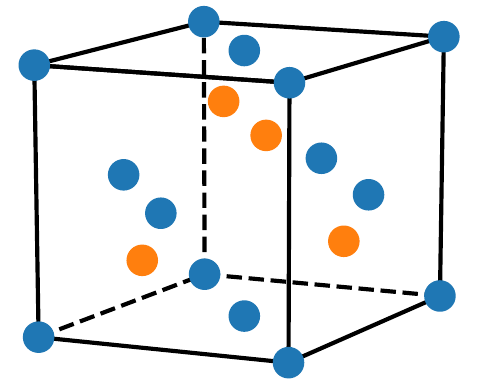}}
	\caption{Examples of multilattice structures.}
	\label{fig:multi}
\end{figure} 

We adopt several equivalent descriptions of the kinematics of the multilattice.  We denote a deformation field by $y_\alpha(\ell) \in \R^3$ and a displacement field by $u_\alpha(\ell) \in \R^3$ for each species $\alpha$.  The argument of both of these fields is a lattice site $\ell \in \mathcal{L}$.  The relationship between the two is
\[
y_\alpha(\ell) = \ell + p^{\rm ref}_\alpha + u_\alpha(\ell).
\]
We collect the set of all deformations (and displacements) at a single lattice site into a tuple 
\[
\bm{y}(\ell) = \big(y_0(\ell), \ldots, y_{S-1}(\ell)\big), \qquad \bm{u}(\ell) = \big(u_0(\ell), \ldots, u_{S-1}(\ell)\big).
\]

We now specify our model to a situation to model straight dislocations, mimicking the setup of~\cite{Ehrlacher2013}. For convenience, we shall assume that the lattice is oriented so that the dislocation direction is parallel to the $x_3$ direction and the Burgers vector is of the form
\[
\mathsf{b} = (\mathsf{b}_1, 0, \mathsf{b}_3) \in \mathcal{L}.
\]
We further assume, without loss of generality, that the displacement fields are independent of the $x_3$-direction and thus only functions of $x_1$ and $x_2$.  We denote the resulting two-dimensional reference lattice by
\[
\Lambda = \mA \mathbb{Z}^2 := \{ (\ell_1, \ell_2): \ell \in \mathcal{L} \}.
\]

Yet another way of describing the multilattice kinematics, motivated by the definition of the multilattice, is by using displacement-shift notation, $(U, \bm{p})$. We emphasize that the shifts ${\bm p}$ in this notation are {\em relative shifts} that are added to the reference shift ${\bm p}^{\rm ref}$. Here, we set
\[
U(\ell) = u_0(\ell), \qquad p_\alpha(\ell) = u_\alpha(\ell) - u_0(\ell), \qquad \bm{p}(\ell) = \big( p_0(\ell), \ldots, p_{S-1}(\ell)\big),
\]
though other choices are available and may in fact be preferable depending on symmetries of the underlying multilattice (see, e.g.,~\cite{koten2013} for one such example).

To describe interactions between atoms, let
\[
\mathcal{R} = \left\{\triple: \rho \in \Lambda, \alpha,\beta \in \mathcal{S}\right\} \setminus \bigcup^{S-1}_{\alpha=0}\{(0\alpha\alpha)\}
\]
be an interaction range allowing us to index pairs of interacting atoms of species $\alpha$ and $\beta$ whose sites are connected by a vector $\rho$, and
\[
D\bm{u}(\ell) := \big(D_{\triple} \bm{u}(\ell)\big)_{(\rho\alpha\beta) \in \mathcal{R}}, \qquad D_{\triple} \bm{u}(\ell) := u_\beta(\ell+\rho) - u_\alpha(\ell)
\]
is the interaction stencil of finite differences needed to compute the energy at site $\ell \in \Lambda$.  We assume $\mathcal{R}$ is finite (this assumption will be justified shortly) and satisfies the conditions
\begin{align}
   & {\rm span}\{ \rho \,|\, (\rho\alpha\alpha) \in \mathcal{R} \} = \R^2 \text{ for all $\alpha \in \mathcal{S}$}, \label{cond1} \\
   & (0\alpha\beta) \in \mathcal{R} \quad \text{for all $\alpha \neq \beta \in \mathcal{S}$ } \label{cond2}.
\end{align}
These two conditions, as well as a further condition encoding slip invariance (see condition~\eqref{cond3}) of the site potential may always be met by enlarging the interaction range if necessary. For future reference, we denote the projection of $\mathcal{R}$ on to the lattice component by
\[
\mathcal{R}_1 = \{\rho \in \Lambda : \triple \in \mathcal{R}\}
\]
and use $r_{\rm cut}$ to denote the cut-off distance for the potential so that if $|\rho| > r_{\rm cut}$, then $\triple \notin \mathcal{R}$.


\subsection{Energy} \label{sec:energy}
Having described the multilattice kinematic variables, we now describe the basic assumptions on the potential energy that we require for our analysis.  These assumptions will be fairly general so as to allow wide-ranging applicability to any classical interatomic potential including in particular general many-body interaction. 
We assume that the site energy potential $V: (\mathbb{R}^3)^{|\mathcal{R}|} \to \mathbb{R}$ is four times continuously differentiable with uniformly bounded derivatives; that is, there exists $M>0$ such that $|\partial_\gamma V| \leq M$ for any multi-index $|\gamma| \leq 4$.  We also assume $V(\bm 0) = 0$ (that is, V describes in fact the energy difference from the reference lattice).

For the atomistic energy functional to be well-defined (i.e., finite), we will consider an energy difference functional defined on displacements $\bm{u}$,
\begin{equation}\label{at_u}
\mathcal{E}^{\rm a}(\bm{u}) := \sum_{\ell \in \Lambda} \Big( V\big(D\bm{u}^0(\ell) + D\bm{u}(\ell)\big) - V\big(D\bm{u}^0(\ell)\big)\Big) =: \sum_{\ell \in \Lambda} V_{\ell}\big(D\bm{u}(\ell)\big)
\end{equation}
where $V_{\ell}({\bf g}) := V\big( D{\bm u}^0(\ell) + {\bf g}\big) - V\big( D {\bm u}^0(\ell)\big)$, for ${\bf g} \in (\mathbb{R}^3)^{|\mathcal{R}|}$, with $\bm{u}^0$ representing a predictor displacement derived from solving a linear elastic model of a dislocation. (This predictor will be defined in~\eqref{cond3}).

An auxiliary energy functional needed in the following analysis is the energy of the homogeneous (defect-free) lattice 
\begin{equation}\label{energy_hom}
\mathcal{E}^{\a}_{\rm hom}(\bm{u}) = \sum_{\ell \in \Lambda} V\big(D\bm{u}(\ell)\big) .
\end{equation}

Arguments of the site potentials are indexed by $(\rho\alpha\beta) \in \mathcal{R}$.
Given $(\rho\alpha\beta), (\tau\gamma\delta) \in \mathcal{R}$, $\ell \in \Lambda$ and $\bm{g} = (g_{\triple})_{\triple \in \mathcal{R}} \in (\mathbb{R}^3)^{\mathcal{R}}$, we then denote the first and second partial derivatives by 
\begin{align*}
[V_{\ell,\triple}(\bm{g})]_{i} :=~& \frac{\partial V_\ell(\bm{g})}{\partial g_{\triple}^i }, \quad i = 1,2,3, \\
V_{\ell,\triple}(\bm{g}) :=~& \frac{\partial V_\ell(\bm{g})}{\partial g_{\triple}}, \\
[V_{\ell,\triple\tripleTau}(\bm{g})]_{ij} :=~& \frac{\partial^2 V_\ell(\bm{g})}{\partial g_{\tripleTau}^j \partial g_{\triple}^i},  \quad i,j = 1,2,3, \\
V_{\ell,\triple\tripleTau}(\bm{g}) :=~& \frac{\partial^2 V_\ell(\bm{g})}{\partial g_{\tripleTau} \partial g_{\triple}},
\end{align*}
with higher order derivatives defined analogously. 


The discrete function space over which $\mathcal{E}^\a_{\rm hom}(\bm{u})$ and $\mathcal{E}^\a(\bm{u})$ will be defined is a quotient space of multilattice displacements whose (semi-)norm,
\[
\| \bm{u}\|_{\a_1}^2 := \sum_{\ell \in \Lambda} |D\bm{u}(\ell)|_\mathcal{R}^2,
 \qquad \text{where} \quad |D\bm{u}|_{\mathcal{R}}^2 := \sum_{\triple \in \mathcal{R}} |D_{\triple} \bm{u}(\ell)|^2,
\]
is finite.  Because this is only a semi-norm, we form the quotient with the kernel of the semi-norm, which is the set of constant multilattice displacements, to obtain the spaces
\begin{align*}
   \bm{\mathcal{U}} := \mathcal{U}/ \mathbb{R}^3, \qquad \text{where}
   \quad \mathcal{U} := \left\{\bm{u}:\mathcal{L}^\mathcal{S}  \to \mathbb{R}^3, \|\bm{u}\|_{\a_1} < \infty   \right\}.
\end{align*}
The elements of $\bm{\mathcal{U}}$ can be understood as equivalence classes, 
\begin{align*}
\bm{\mathcal{U}} = \{ \{(u_{\alpha}+c)_{\alpha\in\mathcal{S}}:c\in\R^3\}:u\in \mathcal{U}\}.
\end{align*}

It is then proven in~\cite{olsonOrtner2016} that $\bm{\mathcal{U}}_0$, defined by
\begin{align*}
\mathcal{U}_0 :=~& \left\{\bm{u} \in \mathcal{U} : Du_0, u_\alpha - u_0 \,
               \mbox{have compact support for each $\alpha \in \mathcal{S}$} \right\}, \\
\bm{\mathcal{U}}_0 :=~& \mathcal{U}_0 / \mathbb{R}^3,
\end{align*}
is dense in $\bm{\mathcal{U}}$, which is used to establish the following Theorem. 

\begin{theorem}[Olson and Ortner 2017 \cite{olsonOrtner2016}]\label{well_defined}
If the homogeneous multilattice reference configuration, $\bm{u} = \bm{0}$, is an equilibrium of the defect free energy, that is, 
\begin{equation}\label{ostrich1}
\sum_{\ell \in \Lambda} \sum_{\triple \in \mathcal{R}} V_{, \triple}(\bm{0}) \cdot D_{(\rho\alpha\beta)}\bm{v}(\ell) = 0, \quad \forall \, \bm{v} \in \bm{\mathcal{U}}_0,
\end{equation}
and if the homogeneous site potential is $C^4$ with uniformly bounded derivatives, then the energy functional, $\mathcal{E}^\a_{\rm hom}(\bm{u})$, is well-defined and $C^3$ (three times continuously Frechet
differentiable) on $\bm{\mathcal{U}}$.
\end{theorem}

\medskip


A further assumption we require on the site potential and multilattice geometry is that the defect-free multilattice is a stable equilibrium of the defect free energy. This is simply a convenient formalization of the usual phonon-stability assumption; see~\cite{olsonOrtner2016,weinan2007cauchy} for an in-depth discussion.  Mathematically, this amounts to a discrete ellipticity condition and in particular allows us to leverage of Green's function estimates for elliptic equations.

\begin{assumption}\label{coercive}
The homogeneous multilattice reference configuration, $\bm{u} = \bm{0}$,
   is a stable equilibrium of $\mathcal{E}^\a_{\rm hom}$ in the sense that there exists $\gamma_\a > 0$
   such that
   \begin{equation} \label{eq:stab-hom}
      \<\delta^2\mathcal{E}^\a_{\rm hom}(\bm{0})\bm{v},\bm{v}\> \geq~ \gamma_\a\|\bm{v}\|_{\a_1}^2 , \quad \forall \, \bm{v} \in \bm{\mathcal{U}}.
   \end{equation}
\end{assumption}

\medskip 

\subsection{Slip Invariance and Admissible Function Space}
\label{sec:slip_space}
We are now in a position to rigorously define the dislocation energy~\eqref{at_u} and the admissible space of displacements over which it is defined.  To that end, we must (1) encode slip invariance of the atomistic energy into the site potential $V$, (2) define the admissible function space as a proper subset of $\bm{\mathcal{U}}$ allowing us to employ a finite interaction range in the reference configuration, and (3) define the predictor displacement $\bm{u}^0$ utilized in~\eqref{at_u}.

We assume a dislocation core position $\hat{x} = (\hat{x}_1, \hat{x}_2) \notin \Lambda$ and assume the dislocation branch cut is given by
\[
\Gamma = \left\{(x_1,x_2) \in \mathbb{R}^2 : x_1 \geq \hat{x}_1, x_2=\hat{x}_2 \right\},
\]
which is consistent with our assumption that the Burgers vector is parallel to
the $x_1$-direction.

In order to model dislocations the homogeneous site potential $V$ must be invariant under lattice slip. Following \cite{Ehrlacher2013} we define the slip operator $S_0$ acting on the multilattice displacements $w_\alpha: \Lambda \rightarrow \mathbb{R}^3$, for $\alpha \in \mathcal{S}$, by 
\[
(S_0w)_\alpha(x) := \begin{cases}
w_\alpha(x), \qquad \qquad \qquad &x_2 > \hat{x}_2, \\
w_\alpha(x-\mathsf{b}_{12}) - \mathsf{b}, \qquad &x_2 < \hat{x}_2,
\end{cases}
\]
where $\mathsf{b}_{12} = ({\sf b}_1, 0)$ is the projection of the Burger's vector to the $(x_1,x_2)$ plane; then we have
\begin{align}\label{cond3}
V(DS_0w(\ell)) = V(Dw(\ell)), \qquad \forall~\ell_2 > \hat{x}_2, \nonumber \\
V(DS_0w(\ell+\mathsf{b}_{12})) = V(Dw(\ell)), \qquad \forall~\ell_2 < \hat{x}_2.
\end{align}
This condition ensures that the energy of the lattice remains invariant under crystallographic slip by a lattice vector.  However, as noted in~\cite{Ehrlacher2013}, in order for this condition to not invalidate our assumption of a finite atomistic interaction range, we must define the admissible displacement space for the dislocation energy~\eqref{at_u} as a proper subset of $\bm{\mathcal{U}}$.

Before introducing that, it is useful to make a connection between the discrete, atomistic objects and the continuous objects. For $u:\mathcal{L} \rightarrow \R^3$, we then denote the continuous interpolant of $u$ with respect to $\mathcal{T}_{\rm a}$ by $Iu$, where $\mathcal{T}_{\rm a}$ is a simplicial decomposition of $\mathcal{L}$ \cite{Ehrlacher2013, olsonOrtner2016} whose vertices are exactly the lattice sites $\ell \in \Lambda$ and we simply take $I$ to be the standard $\mathcal{P}_1$ interpolant. We will also write $I\bm{u} = (I u_{\alpha})_{\alpha\in\mathcal{S}}$.

We can then choose a global bound, $m_{\mathcal{A}}$, for $\|\nabla I u_\alpha\|_{L^\infty}, \|I p_\alpha\|_{L^\infty}$ and a radius $r_{\mathcal{A}}$ large enough so that
\begin{equation}\label{admissible}
\begin{split}
\mathcal{A} := \big\{\bm{u} \in \bm{\mathcal{U}} : \|\nabla Iu_\alpha \|_{L^\infty}, \|I p_\alpha \|_{L^\infty} < m_{\mathcal{A}}, \, \mbox{and} \, |\nabla Iu_\alpha(x) |, |Ip_\alpha(x) | < 1/2, \, x > &r_{\mathcal{A}}, \, \forall \, \alpha \in \mathcal{S}    \big\},
\end{split}
\end{equation}
contains the possible minimizers of the dislocation energy.  Arguing as in~\cite[Appendix B]{Ehrlacher2013}, the key idea here is that the finite energy criterion on $\| \cdot \|_{\a_1}$ in the definition of $\bm{\mathcal{U}}$ implies $\nabla I u_\alpha \to 0$ as $|x| \to \infty$ and similarly for $Ip_\alpha \to 0$.  Thus, $m_{\mathcal{A}}$ and $r_{\mathcal{A}}$ may always be chosen large enough so that a \textit{particular} local minimum of the dislocation energy belongs to $\mathcal{A}$.  But then we may always increase these parameters so that all elements of $\bm{\mathcal{U}}$ within some ball of arbitrary radius about this minimum are also contained in $\mathcal{A}$, which permits us to perform our local calculus arguments. Full details may be found in~\cite[Appendix B]{Ehrlacher2013}.

We may then formulate the slip invariance condition by defining a mapping $S$, where $S$ is an $\ell^2$-orthogonal operator with dual $R=S^{*}=S^{-1}$ \cite{Ehrlacher2013}, of both lattice and multilattice functions by $(S\bm{u})_\alpha(\ell) = Su_\alpha(\ell)$, where
\[
Su(\ell) := \begin{cases} u(\ell), \qquad\qquad &\ell_2 > \hat{x}_2 \\
u(\ell-\mathsf{b}_{12}), \qquad &\ell_2 < \hat{x}_2
\end{cases},  \qquad Ru(\ell) = \begin{cases} u(\ell), \qquad \qquad &\ell_2 > \hat{x}_2 \\
u(\ell+\mathsf{b}_{12}), \qquad &\ell_2 < \hat{x}_2
\end{cases}. \\
\]
Furthermore, for $\bm{\omega} = \bm{u}^0 + \bm{u}, \bm{u}\in\mathcal{A}$, we shall write $S_0 \bm{\omega} = S_0 \bm{u}^0 + S \bm{u}$. 
The slip invariance condition can now be expressed (using the same notation as~\cite{Ehrlacher2013}) as
\begin{equation}\label{slip}
V\big(D(\bm{u}^0 + \bm{u})(\ell)\big) = V\big(RDS_0(\bm{u}^0 + \bm{u})(\ell)\big), \qquad \forall \ell \in \Lambda, \bm{u} \in \mathcal{A}.
\end{equation}

In our analysis we require that applying the slip operator to the predictor map $\bm{u}^0$ yields a smooth function in the half-space.
\[
\Omega_\Gamma = \{(x_1, x_2) \in \mathbb{R}^2 : x_1 \geq \hat{x}_1 \} \setminus B_{\hat{r}+\mathsf{b}_{1}}(\hat{x}),
\]
where the dislocation core radius $\hat{r}$ is defined in Lemma~\ref{ml_thm_vs}. It is therefore natural to define (analogously as in \cite{Ehrlacher2013}) the \textit{elastic strains}
\begin{equation}\label{ml_strain}
e(\ell) := \big(e_{\triple}(\ell)\big)_{\triple \in \mathcal{R}}, \qquad e_{\triple}(\ell) = \begin{cases} RD_{\triple}S_0\bm{u}^0(\ell), \qquad &\ell \in \Omega_{\Gamma} \\
D_{\triple}\bm{u}^0(\ell), \qquad  \quad \quad &\ell \notin \Omega_\Gamma
\end{cases},
\end{equation}
and the analogous definition for corrector $\bm{u}$
\begin{equation}\label{fancy_D}
\tilde{D}\bm{u}(\ell) := \big(\tilde{D}_{\triple}\bm{u}(\ell)\big)_{\triple \in \mathcal{R}}, \qquad \tilde{D}_{\triple}\bm{u}(\ell) = \begin{cases} RD_{\triple}S\bm{u}(\ell), \qquad &\ell \in \Omega_{\Gamma} \\
D_{\triple}\bm{u}(\ell) \qquad \quad \quad &\ell \notin \Omega_\Gamma
\end{cases}.
\end{equation}
Using this notation, the slip invariance condition~\eqref{slip} may be written as
\begin{equation}\label{slip_strain}
V\big(D(\bm{u}^0 + \bm{u})(\ell)\big) = V\big(e(\ell) + \tilde{D}\bm{u}(\ell)\big).
\end{equation}
Moreover, we have the identity
\begin{equation}\label{slip_fd_id}
\tilde{D}_{\triple} \bm{u}(\ell) = \tilde{D}_\rho u_0(\ell) + \tilde{D}_\rho p_\beta(\ell) + p_\beta(\ell) - p_\alpha(\ell),
\end{equation}
where $\tilde{D}_{\rho}$ is defined in \eqref{eq:tildeDrho} later. It is intuitive, but can be proven via tedious algebraic manipulations and considering the cases {\bfseries{(1)}} $\ell \notin \Omega_\Gamma$ and {\bfseries{(2)}} $\ell \in \Omega_\Gamma$.  We have included these manipulations in Appendix~\ref{tedious}.

\subsection{Continuum-Elasticity Dislocation Predictor}
After reviewing the well-posedness of the defect-free energy, $\mathcal{E}^\a_{\rm hom}$, in Theorem~\ref{well_defined} and describing the fundamental assumptions on the energy (smoothness of the site potential and the coercivity condition of Assumption~\ref{coercive}) and the slip invariance condition on the site potential (condition~\eqref{slip_strain}), we are now in a position to complete the definition of the dislocation defect energy first alluded to in~\eqref{at_u}.  Specifically, it remains to define the predictor $\bm{u}^0 = (U^0, \bm{p}^0)$ utilized in~\eqref{at_u}.

As is done in the simple lattice case~\cite{Ehrlacher2013}, this will be accomplished by a slight modification of the solution of a continuum linearized elasticity equation, where the elasticity tensor is taken from the linearized Cauchy--Born model~\cite{cauchy,born1954}.  Thus, we shall define $(U^{\rm lin})$ by
\begin{equation}\label{lin_elastic}
\begin{split}
\nabla \cdot (\mathbb{C} \nabla U^{\rm lin}) =~& 0, \\
U^{\rm lin}(x+) - U^{\rm lin}(x-) =~& -\mathsf{b}, \quad \mbox{on $\Gamma$,} \\
\nabla_{e_2} U^{\rm lin}(x+) - \nabla_{e_2} U^{\rm lin}(x-) =~& 0, \qquad \mbox{on $\Gamma$,}
\end{split}
\end{equation}
where $\mathbb{C}$ is the linearized Cauchy-Born tensor for a multilattice defined by
\begin{equation}\label{elastic_tensor}
\begin{split}
W(\mF, \bm{p}) =~& V\Big(\big(\mF \rho + p_\beta - p_\alpha\big)_{\triple \in \mathcal{R}}\Big) \\
\mathbb{C}_{ijkl} =~& \frac{\partial^2}{\partial \mF_{ij}\partial \mF_{kl}} \min_{\bm{p}} W(\mF, \bm{p}).
\end{split}
\end{equation}

It was shown in~\cite[Equation 3.11]{olsonOrtner2016} that Assumption~\ref{coercive} implies that the linearized Cauchy--Born tensor satisfies a Legendre-Hadamard condition, and therefore the first set of three equations in~\eqref{lin_elastic} is solvable by the classical techniques of Hirth and Lothe~\cite{hirth1982theory}. As we are working with a multilattice, we must then specify also the corresponding {\em relative shifts}, $\bm{p}^{\rm lin}$.  In the Cauchy--Born theory for multilattices, the shift fields are obtained by minimization of the energy density~\cite{born1954}:
\begin{equation}\label{p_min}
\bm{p} = {\argmin} ~~W\big(\nabla U^{\rm lin}, \cdot\big)
\end{equation}
but instead we define $\bm{p}^{\rm lin}$ by minimizing the quadratic expansion of $W$ which we will show is sufficient to maintain the accuracy as the linear elasticity model,
\begin{equation}\label{p_lin_elastic}
\frac{\partial^2 W( {\sf 0}, {\bm 0})}{\partial \bm{p} \partial \bm{p} } (\bm{p}^{\rm lin} - \nabla U^{\rm lin} \bm{p}^{\rm ref}) = -\frac{\partial^2 W({\sf 0}, {\bm 0})}{\partial\bm{p} \partial\mF} \nabla U^{\rm lin}. 
\end{equation}
In the far-field our definition of the shift can be thought of as a leading order correction to a reasonable first guess ${\bm p} = \nabla U^{\rm lin} \bm{p}^{\rm ref}$. 


To finalize the definition of the far-field predictor displacement $\bm{u}^0 = (U^0, \bm{p}^0)$ we must now perform a purely technical modification of $(U^{\rm lin},\bm{p}^{\rm lin})$ near the dislocation core to ``smoothen'' the singularity. To that end we  introduce a smooth transition function, $\eta:\mathbb{R}\to\mathbb{R}$, which satisfies $\eta(x) = 1$ for $x \geq 1$, $\eta(x) = 0$ for $x \leq 0$, and $\eta'(x) > 0$ for $0 < x < 1$; and an argument function $\arg:\mathbb{R}^2 \to [0, 2\pi)$ where $\arg(x)$ is the angle between the positive $x$-axis and $x$.  We then set our predictor displacement and shift fields to be
\begin{equation}\label{predictor_def}
\begin{split}
U^0(x) =~& U^{\rm lin}(\zeta^{-1}(x)), \qquad \zeta(x) := x - b_{12}\eta(|x-\hat{x}|/\hat{r}) \arg(x-\hat{x}),
\end{split}
\end{equation}
and $\bm{p}^0(x)$ is obtained from
\begin{equation}\label{p_predictor_def}
\frac{\partial^2 W({\sf 0}, \bm{0})}{\partial \bm{p} \partial \bm{p} } (\bm{p}^{0} - \nabla U^{0} \bm{p}^{\rm ref}) = -\frac{\partial^2 W({\sf 0}, \bm{0})}{\partial\bm{p} \partial\mF} \nabla U^{0}. 
\end{equation}


The function $\zeta$ is taken verbatim from~\cite{Ehrlacher2013}, where a more detailed discussion of this ``smoothing'' step can be found. In particular, all of the estimates of \cite[Lemma 3.1]{Ehrlacher2013} apply directly to $U^0$ in the present work. Moreover, it is straightforward to deduce corresponding estimates for each $p^0_\alpha$ from equation~\eqref{p_predictor_def} and the estimates for $U^0$, which yields the following result.

\begin{theorem}[Multilattice version of Lemma 3.1 in \cite{Ehrlacher2013}]\label{ml_thm_vs}
There exists a solution, $U^{\rm lin}$, to~\eqref{lin_elastic} and $\hat{r}$ such that $\zeta$ is bijective on $\mathbb{R}^2\setminus\Gamma$. $U^0$ is then well-defined and a solution, $\bm{p}^0$, to~\eqref{p_predictor_def} exists.  These functions satisfy $\nabla^j S_0 u^0_\alpha(x+) = \nabla^j S_0 u^0_\alpha(x-)$ and $\nabla^j p^0_\alpha(x+) = \nabla^j p^0_\alpha(x-)$  for all $\alpha \in \mathcal{S}$, nonnegative integers $j$, and $x \in \Gamma$, and
\[
|\nabla^j U^0(x) - \nabla^j U^{\rm lin}(\zeta^{-1}(x))| \lesssim |x|^{-j-1}, \qquad |\nabla^j p^0_\alpha(x) - \nabla^j p^{\rm lin}_{\alpha}(\zeta^{-1}(x))| \lesssim |x|^{-j-2}.
\]
\end{theorem}

\begin{proof}
As all results concerning displacements are proven in~\cite{Ehrlacher2013}, we shall only concern ourselves with the results for the shifts. Once we establish existence of $\bm{p}$, the corresponding estimates are immediate from the definition of $\bm{p}^0$ and $\bm{p}^{\rm lin}$ in~\eqref{p_lin_elastic} and~\eqref{p_predictor_def} and the corresponding results for $U^0$ and $U^{\rm lin}$.

For existence of a solution to~\eqref{p_predictor_def}, we need only note that it was shown in~\cite[Theorem 3.7]{olsonOrtner2016} that the atomistic stability assumption, Assumption~\ref{coercive}, implies a corresponding estimate on stability of the Cauchy-Born model. In particular, $\frac{\partial^2 W(0, \bm{0})}{\partial \bm{p} \partial \bm{p} }$ was shown to be invertible.
\end{proof}

\subsection{Main Results}

We have now fully defined all ingredients in the dislocation energy~\eqref{at_u}, 
$\mathcal{E}^{\a}(\bm{u}) = \sum_{\ell \in \Lambda} V_\ell(D\bm{u}(\ell))$, 
and employing the techniques of \cite{Ehrlacher2013, olsonOrtner2016} we can further show that this energy is well-defined and continuously Frechet differentiable over the admissible displacements, $\mathcal{A}$, defined in~\eqref{admissible}. 

\begin{theorem}
   \label{energy_thm}
The atomistic energy function, $\mathcal{E}^{\rm a}(\bm{u})$, is well defined and belongs to $C^3(\mathcal{A})$. 
\end{theorem}
\begin{proof}[{\textit{\textbf{Proof (Idea of the proof).}}}] For $\bm{u} \in \bm{\mathcal{U}}_0$, we can write
\begin{equation}
\mathcal{E}^\a(\bm{u}) = \sum_{\ell \in \Lambda}\Big[ V_{\ell}\big(D\bm{u}(\ell)\big) - \< \delta V_\ell(\bm{0}), D\bm{u}(\ell) \>\Big] + \<\delta \mathcal{E}^\a(\bm{0}),\bm{u}\>.
\end{equation}
We will show $ 
\sum_{\ell \in \Lambda}\big[ V_{\ell}(D\bm{u}(\ell)) - \< \delta V_\ell(\bm{0}), D\bm{u}(\ell) \>\big]$ is well-defined and differentiable, while $
\<\delta \mathcal{E}^\a(\bm{0}),\bm{u}\>$ is a bounded linear functional. The details are presented in Section \ref{sec:pf:energy}.
\end{proof}

Having established that $\mathcal{E}^{\rm a}(\bm{u})$ is well-defined and differentiable on the natural energy space $\bm{\mathcal{U}}_0$, we are then interested in the force equilibrium problem
\begin{equation}\label{Aprob}
\<\delta \mathcal{E}^\a(\bm{u}^\infty),\bm{v}\> = 0, \qquad \forall \bm{v} \in \bm{\mathcal{U}}_0.
\end{equation}
We note that two important special cases are local minima (stable equilibria) and index-1 saddles (transition states between stable equilibria). In the present work, we will not go into details about these specific problems but focus on the regularity of equilibria, that is, solutions to \eqref{Aprob}.

Our decay rates are formulated in terms of finite differences (alternatively, they could be written as derivatives of smooth interpolants of lattice functions) using the notation
\begin{equation}
\begin{aligned}
D_{\rho}u_\alpha(\ell) &:= u_\alpha(\ell + \rho) - u_\alpha(\ell),
   \\ 
\tilde{D}_{\rho}u_{\alpha}(\ell) &= \begin{cases} RD_{\rho}Su_{\alpha}(\ell), \quad &\ell \in \Omega_{\Gamma} \\
D_{\rho}u_{\alpha}(\ell), \quad \quad &\ell \notin \Omega_\Gamma,
\end{cases}
\label{eq:tildeDrho}
\end{aligned}
\end{equation}
for $\rho \in \Lambda, ~\alpha \in \mathcal{S}$.

\begin{theorem}[Decay of dislocation far-fields]\label{decay_thm}
Let $\bm{u}^\infty = (U^\infty, \bm{p}^\infty) \in \mathcal{A}$ be a critical point of $\mathcal{E}^\a(\bm{u})$, i.e. satisfy \eqref{Aprob}, and suppose $V \in C^4(\mathcal{A})$, satisfies the slip invariance condition, and the stability Assumption~\ref{coercive} is satisfied.  Then for all $|\ell|$ large enough,
   \begin{equation}\label{decay1}
      \begin{split}
      \big|\tilde{D}_{{\rho}} U^\infty(\ell)\big| \lesssim~& |\ell|^{-2}\log |\ell|,
            \quad \forall {\rho} \in \mathcal{R}_1, \quad \text{and} \\
      \big| p_\alpha^\infty(\ell)\big| \lesssim~& |\ell|^{-2}\log |\ell|,
          \quad \forall \alpha \in \mathcal{S}. 
      \end{split}
   \end{equation}
\end{theorem}
\begin{proof}[{\textit{\textbf{Proof (Idea of the proof).}}}]
The main idea of proving Theorem \ref{decay_thm} is to show that $\bm{u}^\infty$ solves a linearized problem whose Green's function may be estimated in terms of existing Green's function estimates developed in~\cite{olsonOrtner2016} for point defects in multilattices.  The residual terms found in this linearization process are estimated in close analogy to~\cite{Ehrlacher2013}, and then the two estimates are combined to yield the theorem. The complete proof is given in Section \ref{sec:proof:decay_thm}.
\end{proof}


\begin{remark}\label{decay_remark}
   Although we have only stated the decay of the elastic strains, we have no doubt that (a much more tedious and involved argument) would show corresponding decay estimates for strain gradients (and the corresponding shift gradients), 
   \begin{equation}\label{decay2}
      \begin{split}
      \big|\tilde{D}_{\bm{\rho}} U^\infty(\ell)\big| \lesssim~& |\ell|^{-1-j}\log |\ell|,
            \quad \forall~\bm{\rho} \in (\mathcal{R}_1)^j, 1 \leq j \leq 3, \quad \text{and} \\
      \big|D_{\bm{\rho}} p_\alpha^\infty(\ell)\big| \lesssim~& |\ell|^{-2-j}\log |\ell|,
          \quad \forall~\bm{\rho} \in (\mathcal{R}_1)^j, \forall~\alpha \in \mathcal{S}, 0 \leq j \leq 2.
      \end{split}
   \end{equation}
   We remark on how to prove such a result at the end of the proof of Theorem~\ref{decay_thm}, but since we are not providing the details of the proofs we choose not to state this as part of the theorem. These results would require higher regularity of the site potential $V$.
\end{remark}



\section{Applications and Numerical Examples}
\label{sec:numerics}
In this section, we use our main result, Theorem \ref{decay_thm}, to establish convergence rates of a numerical method to approximate the defect core geometry.  We will illustrate the convergence of edge dislocation geometries in a range of diamond cubic and B2 crystal structures. 


\subsection{Supercell approximation scheme}
A simple but effective computational scheme to approximately solve \eqref{Aprob} is to project the problem to a finite-dimensional subspace. Due to the decay estimates \eqref{decay1} and \eqref{decay2}, it is reasonable to expect that simply truncating to a finite domain $\Omega \subset \Lambda$ yields a convergent approximation scheme. We first parametrize the size of $\Omega$ by a radius
\begin{align*}
R_{\Omega} = \sup \{ r > 0 : B_r(\hat{x}) \subset \Omega \}.
\end{align*}
We then define the approximation space
\begin{align*}
\bm{\mathcal{U}}_{\Omega} := \{\bm{u} \in \mathcal{A} : u_\alpha(\ell) = 0, \quad \forall \ell \in \Lambda \setminus \Omega, \alpha \in \mathcal{S}\},
\end{align*}
and solve the finite-dimensional optimisation problem
\begin{equation}\label{uo-prob}
\bm{u}_\Omega = \argmin \big\{ \mathcal{E}^a(\bm{u}) ~\big|~ u \in \bm{\mathcal{U}}_{\Omega}\big\}.
\end{equation}
In other words, all displacements in $\bm{\mathcal{U}}_{\Omega}$ satisfy a boundary condition clamping atom positions outside of $\Omega$ to those obtained from the continuum-elasticity predictor $\bm{u}^0 = (U^0, \bm{p}^0)$.  Since \eqref{uo-prob} is a pure
Galerkin projection of \eqref{Aprob}, and since we have precise estimates on the decay of the exact solution from Theorem~\ref{decay_thm} it is straightforward to prove convergence rates. This result requires a stronger stability condition than Assumption~\ref{coercive} used to establish the decay estimates. 

\begin{assumption}[Strong Stability] \label{strong_stab} There exists $\gamma_{\rm a} > 0$ such that 
\[
    \<\delta^2\mathcal{E}^\a(\bm{u}^\infty)\bm{v},\bm{v}\> \geq~ \gamma_{\rm a}\|\bm{v}\|_{\a_1}^2 , \quad \forall \, \bm{v} \in \bm{\mathcal{U}}_0.
\]
\end{assumption}

\begin{theorem}\label{alg_theorem}
Suppose that $\bm{u}^\infty \in \mathcal{A}$ is a solution to the atomistic Euler-Lagrange equations $\<\delta \mathcal{E}^\a(\bm{u}^\infty),\bm{v}\> = 0$ for all $\bm{v} \in \bm{\mathcal{U}}_0$ and that the strong stability Assumption~\ref{strong_stab} holds. Then there exists $R_0$ such that for all domains $\Omega$ with $R_{\Omega} \geq R_0$, there exists a solution $\bm{u}_\Omega$ to \eqref{uo-prob}, which satisfies
\begin{equation}\label{strain_est}
\|\bm{u}_\Omega - \bm{u}^\infty\|_{\a_1} \lesssim R_{\Omega}^{-1}\log R_{\Omega}.
\end{equation}
\end{theorem}

\begin{proof}
As in any Galerkin method, the key is estimating the best approximation error of $\bm{u}^\infty$ in the space $\bm{\mathcal{U}}_\Omega$ by defining an interpolation operator $\Pi: \bm{\mathcal{U}}_0 \rightarrow \bm{\mathcal{U}}_{\Omega}$. This was accomplished in the work~\cite[Lemma A.1]{olsonOrtner2016} and~\cite[Lemma 12]{multiBlendingArxiv}, and it was shown that
\begin{align*}
\|\bm{u}^\infty - \Pi \bm{u}^\infty\|_{\a_1} \lesssim \|\nabla I U^{\infty}\|_{L^2(\mathbb{R}^2\setminus B_{R_{\Omega}/2}(\hat{x}))} + \|I\bm{p}^{\infty}\|_{L^2(\mathbb{R}^2\setminus B_{R_{\Omega}/2}(\hat{x}))},
\end{align*}
where we recall $I$ was a piecewise linear interpolation operator defined in Section \ref{sec:slip_space}. Using the decay estimates in Theorem~\ref{decay_thm}, we may then estimate these last two terms and have
\begin{align*}
\|\bm{u}^\infty - \Pi \bm{u}^\infty\|_{\a_1} \lesssim~ R_{\Omega}^{-1}\log R_{\Omega}.
\end{align*}
By continuity, Assumption~\ref{strong_stab} implies the strong stability on $\Pi \bm{u}^\infty$
\begin{align*}
\<\delta^2 \mathcal{E}^\a(\Pi \bm{u}^\infty)\bm{v},\bm{v} \> \gtrsim \|\bm{v}\|_{\a_1}^2, \quad \forall \bm{v} \in \bm{\mathcal{U}}_0,
\end{align*}
with a different implied constant that depends on $\gamma_{\rm a}$.  Moreover, we have
\begin{align*}
\<\delta \mathcal{E}^\a(\Pi \bm{u}^\infty), \bm{v} \> =~& \<\delta \mathcal{E}^\a(\Pi \bm{u}^\infty), \bm{v} \> - \<\delta \mathcal{E}^\a(\bm{u}^\infty), \bm{v} \> \\
\lesssim~& \|\Pi \bm{u}^\infty - \bm{u}^\infty\|_{\a_1} \|\bm{v}\|_{\a_1},  \quad \forall \bm{v} \in \bm{\mathcal{U}}_0,
\end{align*}
so it allows us to apply the inverse function theorem to deduce the existence of a solution $\bm{u}_\Omega \in \bm{\mathcal{U}}_\Omega$ to \eqref{uo-prob}
which satisfies
\[
\|\bm{u}_\Omega - \Pi \bm{u}^\infty\|_{\a_1} \lesssim~ R_{\Omega}^{-1}\log R_{\Omega}.
\]
Hence, we can obtain
\begin{align*}
\|\bm{u}_\Omega -  \bm{u}^\infty\|_{\a_1} \leq \|\bm{u}_\Omega - \Pi \bm{u}^\infty\|_{\a_1} + \|\bm{u}^\infty - \Pi \bm{u}^\infty\|_{\a_1} \lesssim~ R_{\Omega}^{-1}\log R_{\Omega},
\end{align*}
which yields the stated result.
\end{proof}

\subsection{Edge Dislocation in Diamond Structures}
%
%
Multilattice crystals that have diamond structure can be regarded as the combination of two interpenetrating FCC lattices, where in the definition of a multilattice, we can take:
\[
\mB = \frac{1}{2} \begin{pmatrix} 0 &1 &1 \\
                      1 &0 &1 \\
											1 &1 &0 \end{pmatrix}, \quad p_0 = \bm{0}, \quad p_1 = \begin{pmatrix}1/4 \\ 1/4 \\ 1/4 \end{pmatrix}.
\]
We will test the predictions of Theorem~\ref{alg_theorem} on the computation of an edge disloction in Si and in CdTe, both modelled by a Stillinger-Weber potential~ \cite{stillinger1985}. We use the implementation provided by the open-source interatomic potential library \textsf{OpenKIM}~\cite{tadmor2011potential}. Following~ \cite{nandedkar1987atomic} we consider an edge dislocation with Burger's vector $\mathsf{b} = \frac{1}{2}\<\bar{1} 1 0\>$ in the $\{\bar{1} \bar{1} 1\}$ slip plane such that the dislocation line has direction $\<1 1 2\>$. The units have been chosen so that the lattice constant is unity. The corresponding edge dislocation in CdTe is shown in Figure \ref{fig:Disloc_CdTe}. The analytical solution to the continuum elasticity predictor problem may then be obtained from~\cite{hirth1982theory} and the predictor shifts from our linear approximation~\eqref{p_predictor_def}.

\begin{figure}
{\resizebox{3.5in}{2.4in}{\includegraphics{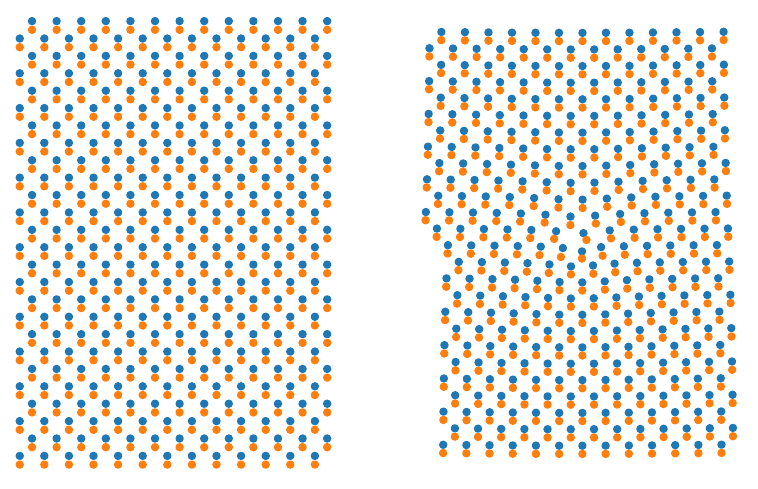}}}
\caption{ (110)[100] Edge Dislocation in CdTe (different colors represent different chemical species). Left: homogeneous crystal; Right: configuration after applying the predictor displacement ${\bm u}^0$.}
\label{fig:Disloc_CdTe}
\end{figure}

We perform a convergence study for the error in the $\|\cdot\|_{\a_1}$ norm by increasing the computational domain size $R_{\Omega}$. In Figure \ref{fig1} we show the decay of the geometry error $\|\bm{u} - \bm{u}_{\Omega}\|_{\a_1}$, where the reference solution $\bm{u}$ against which the error is measured is taken to be a numerical solution on a much larger domain with radius $R_{\rm dom} = 100$. The results are shown in Figure~\ref{fig1} and confirm the first order decay predicted by Theorem \ref{alg_theorem}. We also show the results if a pure continuum linearized elasticity (CLE) predictor without shift relaxation is employed. Instead, the {\em relative shift} for the far-field predictor is simply obtained from an affine transform following the CLE field, 
\begin{equation} \label{eq:naive_shift_predictor}
    {\bm p} = \nabla U^0 {\bm p}^{\rm ref},
\end{equation}
which is the first guess of solving \eqref{p_predictor_def}.
A lower order  convergence can be observed for both the Si and CdTe cases. It is not entirely clear to us why any convergence is obtained at all, but a possible explanation is that the naive guess \eqref{eq:naive_shift_predictor} is already a good approximation to the fully relaxed shift. 

\begin{figure}
    {\resizebox{4.0in}{3.0in}{\includegraphics{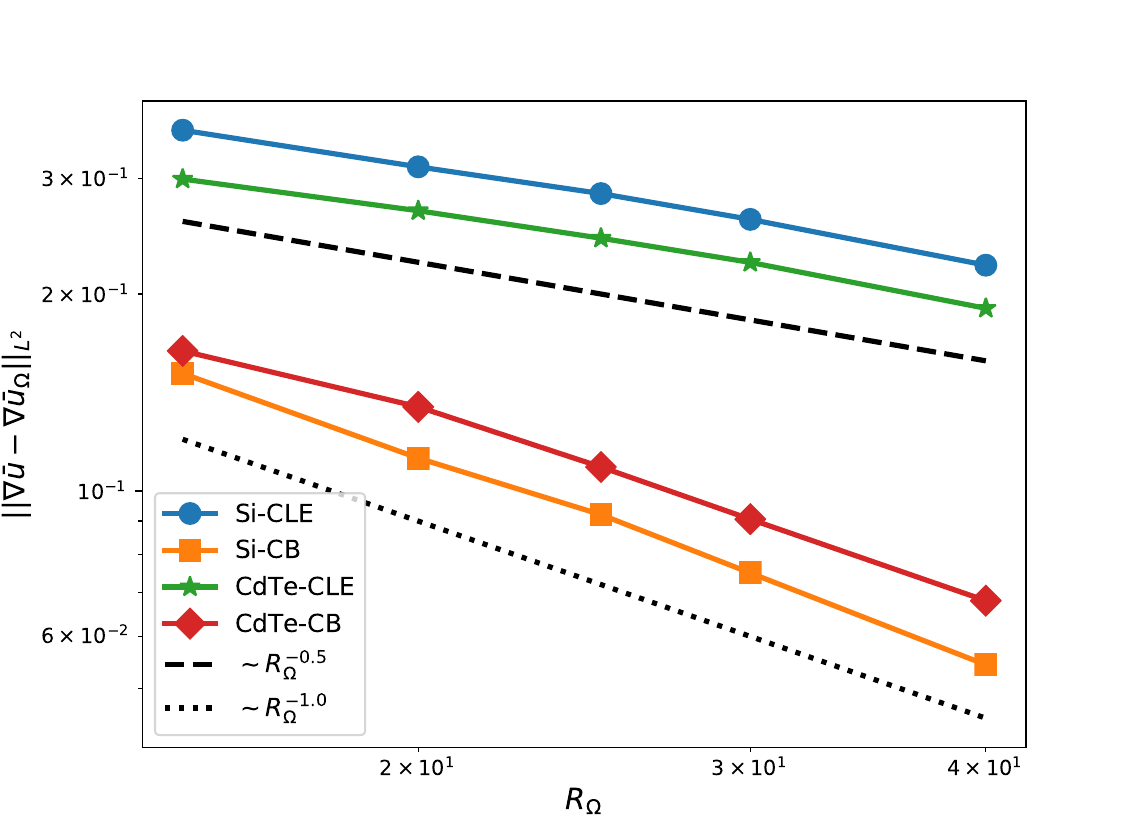}}}
    \caption{Geometry error for edge dislocation in Si and CdTe using Stilinger-Weber type potential.}
    \label{fig1}
\end{figure}

\begin{figure}
    {\resizebox{4.0in}{3.0in}{\includegraphics{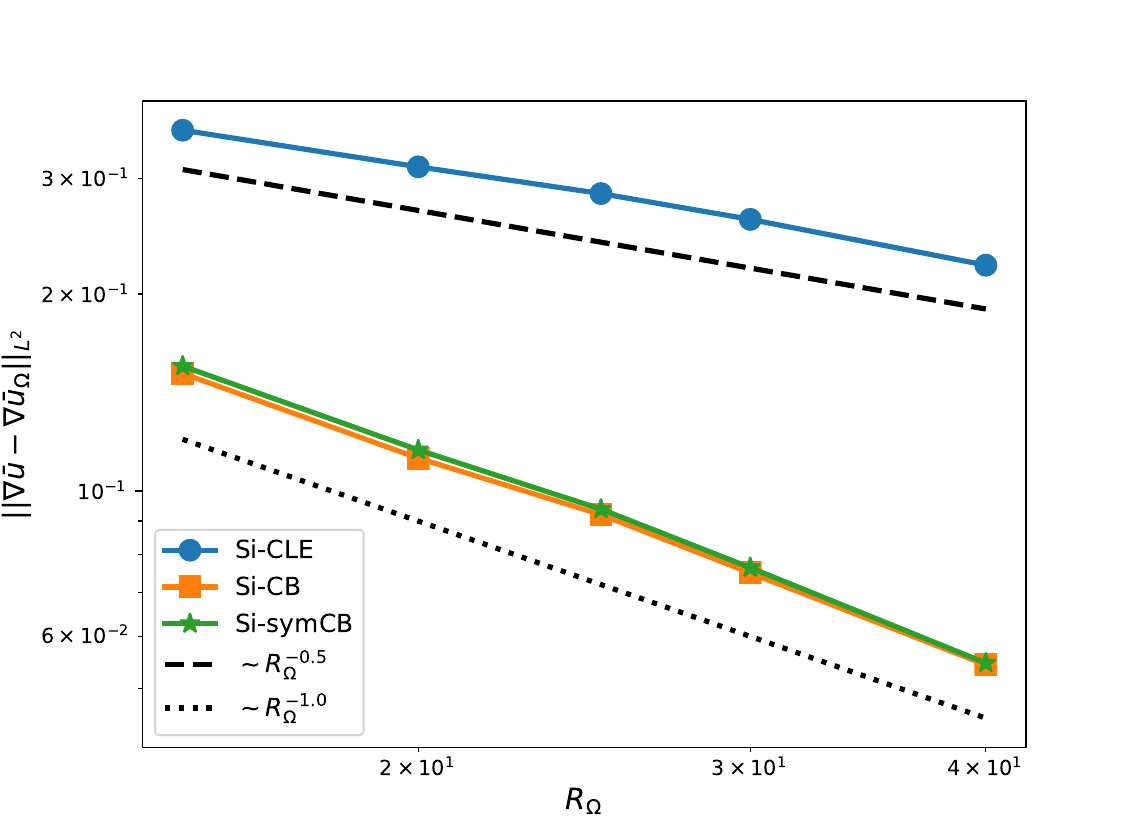}}}
    \caption{Geometry error for edge dislocation in Si using CLE, Cauchy-Born and symmetrized Cauchy-Born predictors.}
    \label{fig:symCB}
\end{figure}

In addition, we also test the effect of using the symmetrized Cauchy-Born rule proposed in \cite{koten2013} as the far-field predictor. In this scheme, the lattice site is placed in the center of a bond ${\bm m}$ instead of an atom. A single shift variable is then using to predict the atom positions at ${\bm m} \pm \frac12 {\bm p}$. This scheme can only be applied in single-species 2-lattices where additional symmetries are available, in particular diamond cubic Si; see \cite{koten_2011} for more details. The results of using the symmetrized Cauchy-Born rule are shown in Figure \ref{fig:symCB}. The convergence matches the regular Cauchy-Born rate very closely, despite the fact that the symmetrized Cauchy-Born energy does not employ an additional relaxation step. 

\subsection{Edge Dislocation in B2 Structures}
The B2 structure is an ordered bcc structure consisting of two simple cubic (SC) interpenetrating sublattices, and stoichiometrically it can be represented by 50:50 atomic distributions. Multilattice crystals forming B2 structure include for example nickel aluminum (AlNi), iron aluminum (AlFe), cobalt aluminum (AlCo) which we choose for our numerical tests. To define the multilattice, we can take:
\[
\mB = \begin{pmatrix} 1 &0 &0 \\
                      0 &1 &0 \\
											0 &0 &1 \end{pmatrix}, \quad p_0 = \bm{0}, \quad p_1 = \begin{pmatrix}1/2 \\ 1/2 \\ 1/2 \end{pmatrix}.
\]
We model the AlNi, AlFe and AlCo materials using embedded atom method (EAM) potentials \cite{Daw1984a} from the \textsf{OpenKIM} library. We consider an edge dislocation with Burger's vectors $\mathsf{b} = \<1 0 0\>$ in the $\{0 0 1\}$ slip plane. The corresponding geometries are shown in Figure \ref{fig:Disloc_NiAl}.

\begin{figure}
{\resizebox{4.5in}{2.25in}{\includegraphics{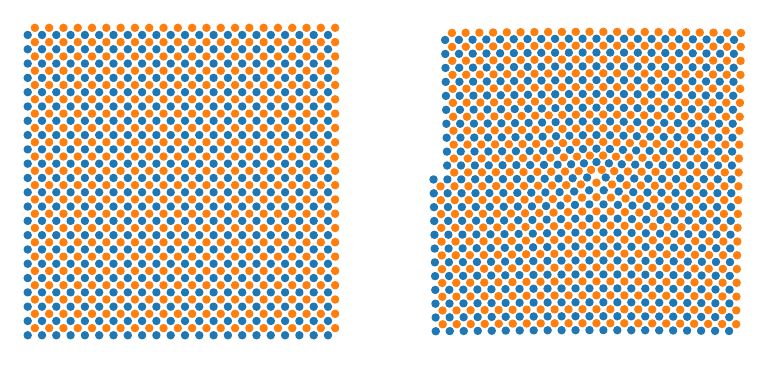}}}
\caption{ (001)[100] Edge Dislocation in B2 structures (different colors denote different chemical species). Left: homogeneous crystal; Right: deformed geometry after applying the predictor ${\bm u}^0$. \label{fig:Disloc_NiAl}}
\end{figure}

 We perform a convergence study for the error measured in the $\|\cdot\|_{\a_1}$ norm with increasing computational domain size $R_{\Omega}$. As reference solution we take ${\bm u}$ solved on a domain with radius $R_{\rm dom}=120$. The results are shown in Figure~\ref{fig2}. All solutions with Cauchy--Born predictor exhibit first order convergence as expected from the theory. Interestingly, using the simpler CLE predictor we now also obtain the quasi-optimal linear rate. This is an artifact of the high symmetry of the B2 structures: due to this high symmetry the simple CLE predictor of the shifts \eqref{eq:naive_shift_predictor} is already the optimal shift. All our numerical results are therefore entirely consistent with our theory. 
 

\begin{figure}[htp!]
{\resizebox{4.0in}{3.0in}{\includegraphics{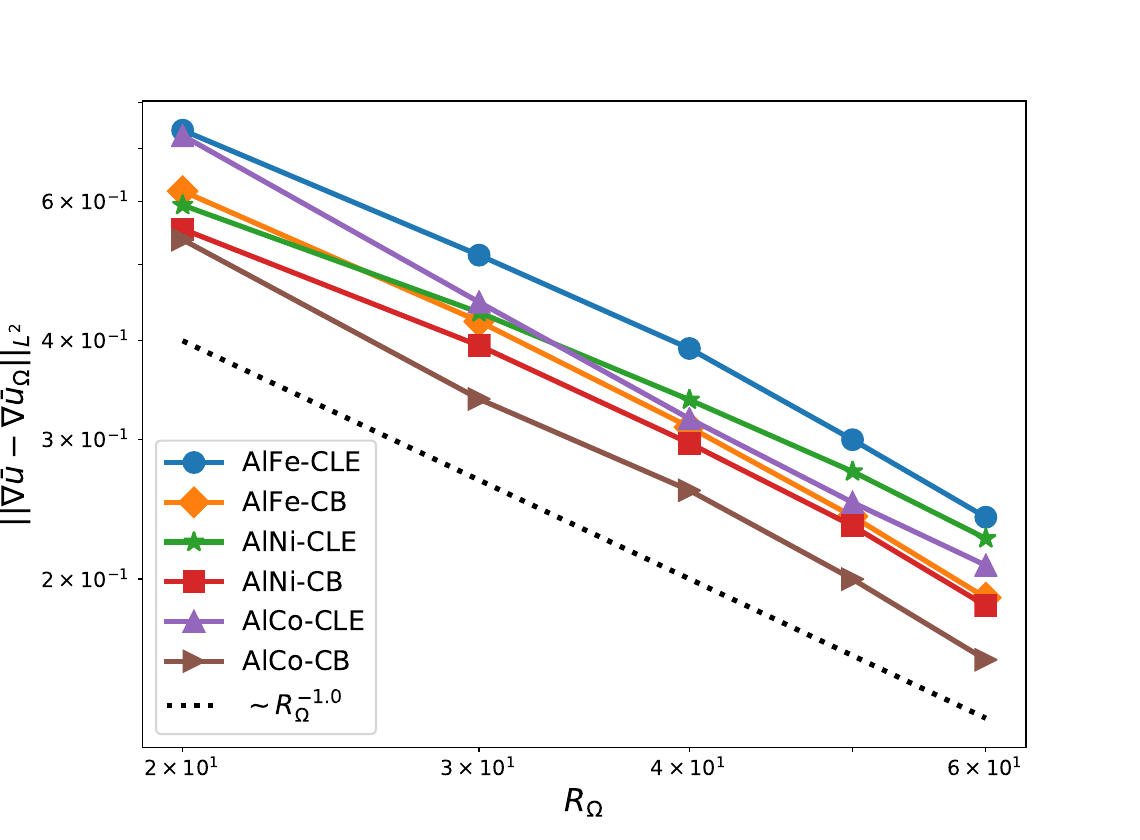}}}
\caption{Geometry error for edge dislocation in BCC crystals (AlNi, AlFe and AlCo) using EAM type potential.}
\label{fig2}
\end{figure}

\section{Conclusion}
In this work we extended the analysis of~\cite{Ehrlacher2013} characterising the decay of discrete elastic fields generated by defects in crystalline materials to general straight dislocation lines in multilattices. Specifically, we established that the elastic field generated by a dislocation in a multilattice can be decomposed into a continuum field predicted by a linearised Cauchy--Born elasticity theory, with a linearised discrete shift correction, and a discrete and nonlinear core corrector representing the defect core. We also discussed in detail the consequences of this result for cell size effects in numerical simulations.

\section{Proofs: The Energy Difference Functionals}
\label{sec:proof:energy_thm}

We first give a useful Lemma, which essentially follows from the Taylor expansions of finite differences using $\tilde{D}$ operator defined by \eqref{fancy_D}, due to $(S_0 U^0, \bm{p}^0)$ being smooth. It plays a key role in proving the decay estimate of residual forces (Lemma \ref{at_force}).


\begin{lemma}\label{diff_grad}
With $\bm{u}^0 = (U^0, \bm{p}^0)$ and $e$ defined by \eqref{predictor_def}, \eqref{p_predictor_def} and \eqref{ml_strain} correspondingly, $\tilde{D}$ is defined by \eqref{fancy_D}, for $(\rho\alpha\beta) \in \mathcal{R}$ and $|\ell|$ large enough, we have
\begin{align}
\big|\tilde{D}_{\triple} \bm{u}^0(\ell) - \big(\nabla_{\rho} U^0(\ell) + p^0_\beta(\ell) - p^0_\alpha(\ell)\big)\big|  \lesssim~& |\ell|^{-2}, \quad \ell \notin \Omega_\Gamma, \label{zero_diff}\\
\big|\tilde{D}_{\triple} \bm{u}^0(\ell) - \big(R\nabla_{\rho} S_0U^0(\ell) + p^0_\beta(\ell) - p^0_\alpha(\ell)\big)\big|  \lesssim~& |\ell|^{-2}, \quad \ell \in \Omega_\Gamma
, \label{first_diff} \\
\big| e(\ell)\big| \lesssim~& |\ell|^{-1}, \label{second_diff} \\
\big|\tilde{D}_{-\rho} e(\ell)\big| \lesssim~& |\ell|^{-2} \label{third_diff}.
\end{align}
\end{lemma}
\begin{proof}
We first prove~\eqref{zero_diff} and consider the two cases as to whether $\ell \in \Omega_\Gamma$.  If $\ell \notin \Omega_\Gamma$, then $\tilde{D} = D$, $\bm{u}^0=(U^0,\bm{p}^0)$ is smooth, and we have a straightforward Taylor expansion
\begin{align*}
&\tilde{D}_{\triple} \bm{u}^0(\ell) =~ D_{\triple} \bm{u}^0(\ell) =~ D_{\rho} U^0(\ell) + D_{\rho}p^0_\beta(\ell) + p^0_\beta(\ell) - p^0_\alpha(\ell) \\
&=~ \nabla_{\rho} U^0(\ell) + \nabla_{\rho} p^0_\beta(\ell) + p^0_\beta(\ell) - p^0_\alpha(\ell) + \mathcal{O}(\|\nabla^2 U^0\|_{L^\infty(B_{r_{\rm cut}}(\ell))}) + \mathcal{O}(\|\nabla^{2} p^0_\beta\|_{L^\infty(B_{r_{\rm cut}}(\ell))}),
\end{align*}
and now from~\cite{hirth1982theory}, we know that $|\nabla^2 U^{\rm lin}| \lesssim~ |\ell|^{-2}$ and thus $|\nabla^2 U^{0}| \lesssim~ |\ell|^{-2}$ from~Theorem~\ref{ml_thm_vs} which further implies $|\nabla \bm{p}^0(\ell)| \lesssim~ |\ell|^{-2}$ from \eqref{p_predictor_def}. Hence
\[
\big|\tilde{D}_{\triple} \bm{u}^0(\ell) - \big(\nabla_{\rho} U^0(\ell) + p^0_\beta(\ell) - p^0_\alpha(\ell)\big)\big|  \lesssim~ |\ell|^{-2}.
\]

If $\ell \in \Omega_\Gamma$, from the definition  \eqref{ml_strain},  \eqref{fancy_D} and Appendix \ref{tedious}, we have 
\begin{align*}
\tilde{D}_{\triple} \bm{u}^0(\ell) &=~ \tilde{D}_\rho U^0(\ell) + \tilde{D}_\rho p^0_\beta(\ell) + p^0_\beta(\ell) - p^0_\alpha(\ell) \\
&=~ RD_\rho S_0 U^0(\ell) + p^0_\beta(\ell_1) + p^0_\beta(\ell_2) - p^0_\beta(\ell) - p^0_\alpha(\ell)
\end{align*}
for $\ell_1 = \ell+\rho\pm \mathsf{b}_{12}$, $\ell_2 = \ell \pm \mathsf{b}_{12}$ by definition of $\tilde{D}$ applied to $\bm{p}^0$. But now both $S_0U^0$ and $\bm{p}^0$ are smooth so we again have a Taylor expansion of
\begin{align*}
&\tilde{D}_{\triple} \bm{u}^0(\ell)
= R\nabla_{\rho}S_0U^0(\ell) + p^0_\beta(\ell) - p^0_\alpha(\ell) + \mathcal{O}(\|\nabla^2 S_0U^0\|_{L^\infty(B_{r_{\rm cut}}(\ell))} + \mathcal{O}(\|\nabla p^0_\beta\|_{L^\infty(B_{r_{\rm cut}}(\ell))}),
\end{align*}
which, because $S_0$ simply represents a shift by one Burgers vector, implies
\begin{align*}
\big|\tilde{D}_{\triple} \bm{u}^0(\ell) - \big(R\nabla_{\rho}S_0U^0 (\ell) + p_\beta(\ell) - p_\alpha(\ell)\big)\big| 
\lesssim~|\ell|^{-2}.
\end{align*}

In order to prove~\eqref{second_diff}, we consider only the case $\ell \notin \Omega_\Gamma$ as the other case is analogous:
\begin{align*}
\big|e(\ell)\big| \leq~& \big|\tilde{D}_{\triple} \bm{u}^0(\ell) - \big(\nabla_{\rho} U^0(\ell) + p^0_\beta(\ell) - p^0_\alpha(\ell)\big)\big| + \big|\nabla_{\rho} U^0(\ell) + p^0_\beta(\ell) - p^0_\alpha(\ell)\big| \\
\lesssim~& |\ell|^{-2} + \big|\nabla_{\rho} U^0(\ell) + p^0_\beta(\ell) - p^0_\alpha(\ell)\big|.
\end{align*}
From~\cite{hirth1982theory}, we know that $|\nabla U^{\rm lin}| \lesssim~ |\ell|^{-1}$ and thus $|\nabla U^{0}| \lesssim |\ell|^{-1}$from Theorem~\ref{ml_thm_vs} which further implies $|\bm{p}^0(\ell)| \lesssim~ |\ell|^{-1}$ from \eqref{p_predictor_def}. Hence we have
\[
\big|e(\ell)\big| \lesssim~ |\ell|^{-1}.
\]
For~\eqref{third_diff}, we apply Taylor remainder estimates analogous to those proven in~\eqref{zero_diff} and~\eqref{first_diff}, and for brevity, we consider only the case $\ell \in \Omega_\Gamma$:
\begin{align*}
\big|\tilde{D}_{-\rho} e(\ell)\big| =~& \big| \tilde{D}_{-\rho}\big( RD_\rho S_0 U^0(\ell) + \tilde{D}_\rho p^0_\beta(\ell) + p_\beta^0(\ell) - p_\alpha^0(\ell)\big) \big| \\
=~& \big| RD_{-\rho}D_{\rho}S_0U^0(\ell) + RD_{-\rho}D_{\rho} S_0 p^0_\beta(\ell) + \tilde{D}_{-\rho}p_\beta^0(\ell) - \tilde{D}_{-\rho}p_\alpha^0(\ell) \big| \\
=~& \big| R\nabla^2_\rho S_0U^0(\ell) + \nabla_{-\rho} p_\beta^0(\ell) - \nabla_{-\rho}p_\alpha^0(\ell) \big| + \mathcal{O}(\|\nabla^3 S_0U^0\|_{L^\infty(B_{r_{\rm cut}}(\ell))}) \\
&\quad + \mathcal{O}(\|\nabla^2 p^0_\beta\|_{L^\infty(B_{r_{\rm cut}}(\ell))}) + \mathcal{O}(\|\nabla^2 p^0_\alpha\|_{L^\infty(B_{r_{\rm cut}}(\ell))})
\end{align*}

We may again utilize~\cite{hirth1982theory}, to see that $|\nabla^2 U^{\rm lin}| \lesssim |\ell|^{-2}$ and $|\nabla^3 U^{\rm lin}| \lesssim |\ell|^{-3}$  and thus $|\nabla^2 S_0U^{0}| \lesssim |\ell|^{-2}$ and $|\nabla^3 S_0U^{0}| \lesssim |\ell|^{-3}$ from~Theorem~\ref{ml_thm_vs}. Obtaining corresponding estimates on the shifts from \eqref{p_predictor_def} gives~\eqref{third_diff}.

\end{proof}

\subsection{Decay Estimate of the Residual Forces}

We then estimate the decay of the residual forces from the homogeneous energy when evaluated at the continuum dislocation predictor; these decay estimates will turn out to be vital in proving the decay estimates for the dislocation strain fields themselves. For ease of notation and visual clarity, throughout the proof, we will use the notation $\|\cdot\|_{L^\infty(\nu_{\ell})}$ to represent $\|\cdot\|_{L^\infty(B_{r_{\rm cut}+\mathsf{b}_1}(\ell))}$, and $\displaystyle\sum^\gamma$ will represent a summation over which $\gamma$ is held fixed. 

\begin{lemma}\label{at_force}
Suppose that $(U,\bm{p})$ are smooth. The force on an atomistic degree of freedom $(\ell,\gamma)$ for $\ell \in \Lambda$ and $\gamma \in \mathcal{S}:=\{0, \ldots, S-1\}$ is given by
\begin{equation}\label{at_force_eq}
\begin{split}
\frac{\partial \mathcal{E}_{\rm hom}^\a(\bm{u})}{\partial u_\gamma(\ell)}\Big|_{(U,\bm{p})} =~& \overbrace{\sum_{(\tau\iota\chi)} \sum_{(\rho\alpha\gamma)}^\gamma V_{,(\rho\alpha\gamma)(\tau\iota\chi)}(\bm{0}) \big(\nabla_{\tau,-\rho}^2 U(\ell) - \nabla_{\rho}p_\chi(\ell) + \nabla_{\rho} p_\iota(\ell) \big)}^{1} \\
&+~ \overbrace{\sum_{(\tau\iota\chi)} \Big(\sum_{(\rho\alpha\gamma)}^\gamma V_{,(\rho\alpha\gamma)(\tau\iota\chi)}(\bm{0}) -  \sum_{(\rho\gamma\alpha)}^\gamma V_{,(\rho\gamma\alpha)(\tau\iota\chi)}(\bm{0})\Big)\big( \nabla_\tau U(\ell) + p_\chi(\ell) - p_\iota(\ell)  \big)}^{2a} \\
&+~ \underbrace{\sum_{(\tau\iota\chi)} \Big(\sum_{(\rho\alpha\gamma)}^\gamma V_{,(\rho\alpha\gamma)(\tau\iota\chi)}(\bm{0}) -  \sum_{(\rho\gamma\alpha)}^\gamma V_{,(\rho\gamma\alpha)(\tau\iota\chi)}(\bm{0})\Big)\big(\frac{1}{2}\nabla^2_{\tau}U(\ell) + \nabla_\tau p_\chi(\ell)  \big)}_{2b}  \\
&+~ \mathcal{O}\big(\|\nabla^2 U\|_{L^\infty(\nu_{\ell})}\|\nabla u\|_{L^\infty(\nu_{\ell})} + \|\nabla \bm{p}\|_{L^\infty(\nu_{\ell})}\|\bm{p}\|_{L^\infty(\nu_{\ell})} + \|\nabla^3 U\|_{L^\infty(\nu_{\ell})} + \|\nabla^2 \bm{p}\|_{L^\infty(\nu_{\ell})} \big).
\end{split}
\end{equation}
In particular, if $(U, \bm{p}) = (U^0, \bm{p}^0)$ solved by \eqref{predictor_def} and \eqref{p_predictor_def}, for $\ell \notin \Omega_\Gamma$ satisfying $|\ell|$ large enough, then
\begin{equation}\label{at_force_lin}
\begin{split}
\sum_{\gamma} \frac{\partial \mathcal{E}^\a_{\rm hom}(\bm{u})}{\partial u_\gamma(\ell)}\Big|_{(U^0,\bm{p}^0)} =~& \mathcal{O}\big(|\ell|^{-3}\big),
\end{split}
\end{equation}
and for $\ell \in \Omega_\Gamma$, if $(U, \bm{p}) = (SU^0, S\bm{p}^0)$, then
\begin{equation}\label{at_force_lin_s}
\begin{split}
\sum_{\gamma} \frac{\partial \mathcal{E}^\a_{\rm hom}(\bm{u})}{\partial u_\gamma(\ell)}\Big|_{(SU^0, S\bm{p}^0)} =~& \mathcal{O}\big(|\ell|^{-3}\big).
\end{split}
\end{equation}
\end{lemma}

\begin{proof}
First, we establish the expression \eqref{at_force_eq}. From definition \eqref{energy_hom}, we observe that
\begin{align}\label{eq:ehom_uy}
\frac{\partial \mathcal{E}^\a_{\rm hom}(\bm{u})}{\partial u_\gamma(\eta)}\Big|_{(U,\bm{p})} =~& \sum_{(\rho\alpha\gamma)}^\gamma V_{,(\rho\alpha\gamma)}(D\bm{u}(\ell-\rho)) - \sum_{(\rho\gamma\alpha)}^\gamma V_{,(\rho\gamma\alpha)}(D\bm{u}(\ell)).
\end{align}
We then Taylor expand the right hand side of \eqref{eq:ehom_uy} at $\bm{0}$ to obtain
\begin{align}\label{taylor1}
&\frac{\partial \mathcal{E}^\a_{\rm hom}(\bm{u})}{\partial u_\gamma(\ell)}\Big|_{(U,\bm{p})} \nonumber \\
&=~ \sum_{(\rho\alpha\gamma)}^\gamma V_{,(\rho\alpha\gamma)}(\bm{0}) -  \sum_{(\rho\gamma\alpha)}^\gamma V_{,(\rho\gamma\alpha)}(\bm{0}) \nonumber \\
&~+ \sum_{(\rho\alpha\gamma)}^\gamma\sum_{(\tau\iota\chi)} V_{,(\rho\alpha\gamma)(\tau\iota\chi)}(\bm{0})\big(D_{(\tau\iota\chi)}\bm{u}(\ell-\rho)\big) - \sum_{(\rho\gamma\alpha)}^\gamma\sum_{(\tau\iota\chi)} V_{,(\rho\gamma\alpha)(\tau\iota\chi)}(\bm{0})\big(D_{(\tau\iota\chi)}\bm{u}(\ell)\big) \nonumber \\
&~ +\frac{1}{2}\int_0^1 (1-t)^2\sum_{(\rho\alpha\gamma)}^\gamma\sum_{(\tau\iota\chi)}\sum_{(\sigma\mu\nu)} V_{,(\rho\alpha\gamma)(\tau\iota\chi)(\sigma\mu\nu)}(tD\bm{u})\big(D_{(\tau\iota\chi)}\bm{u}(\ell-\rho)\big)\big(D_{(\sigma\mu\nu)}\bm{u}(\ell-\rho)\big) \, dt  \nonumber \\
&~- \frac{1}{2}\int_0^1 (1-t)^2\sum_{(\rho\gamma\alpha)}^\gamma\sum_{(\tau\iota\chi)}\sum_{(\sigma\mu\nu)} V_{,(\rho\gamma\alpha)(\tau\iota\chi)(\sigma\mu\nu)}(tD\bm{u})\big(D_{(\tau\iota\chi)}\bm{u}(\ell)\big)\big(D_{(\sigma\mu\nu)}\bm{u}(\ell)\big) \, dt  \nonumber \\
&=~ \sum_{(\rho\alpha\gamma)}^\gamma\sum_{(\tau\iota\chi)} V_{,(\rho\alpha\gamma)(\tau\iota\chi)}(\bm{0})\big(D_{(\tau\iota\chi)}\bm{u}(\ell-\rho)\big) - \sum_{(\rho\gamma\alpha)}^\gamma\sum_{(\tau\iota\chi)} V_{,(\rho\gamma\alpha)(\tau\iota\chi)}(\bm{0})\big(D_{(\tau\iota\chi)}\bm{u}(\ell)\big) \nonumber \\
&~ + \mathcal{O}\big(\|\nabla^2 U\|_{L^\infty(\nu_{\ell})}\|\nabla u\|_{L^\infty(\nu_{\ell})} + \|\nabla \bm{p}\|_{L^\infty(\nu_{\ell})}\|\bm{p}\|_{L^\infty(\nu_{\ell})}\big),
\end{align}
where in obtaining the last line we have used that $\partial p_\gamma W(0, {\bm 0}) = 0$ in the equilibrated reference configuration, which implies $0 = \sum_{(\rho\alpha\gamma)}^\gamma V_{,(\rho\alpha\gamma)}(\bm{0}) -  \sum_{(\rho\gamma\alpha)}^\gamma V_{,(\rho\gamma\alpha)}(\bm{0})$~\cite[A.4]{olsonOrtner2016}, and we Taylor expand the finite differences in the remainder term. Next, we rewrite \eqref{taylor1} as
\begin{align}\label{taylor2}
&\frac{\partial \mathcal{E}^\a_{\rm hom}(\bm{u})}{\partial u_\gamma(\ell)}\Big|_{(U,\bm{p})} \nonumber \\
&=~ \sum_{(\tau\iota\chi)} \Big( \sum_{(\rho\alpha\gamma)}^\gamma V_{,(\rho\alpha\gamma)(\tau\iota\chi)}(\bm{0})\big(D_{(\tau\iota\chi)}\bm{u}(\ell-\rho) - D_{(\tau\iota\chi)}\bm{u}(\ell)  \big)  + \sum_{(\rho\alpha\gamma)}^\gamma V_{,(\rho\alpha\gamma)(\tau\iota\chi)}(\bm{0})D_{(\tau\iota\chi)}\bm{u}(\ell)  \nonumber \\
&~- \sum_{(\rho\gamma\alpha)}^\gamma V_{,(\rho\gamma\alpha)(\tau\iota\chi)}(\bm{0}) D_{(\tau\iota\chi)}\bm{u}(\ell) 
 \Big) + \mathcal{O}\big(\|\nabla^2 U\|_{L^\infty(\nu_\ell)}\|\nabla u\|_{L^\infty(\nu_\ell)} + \|\nabla \bm{p}\|_{L^\infty(\nu_\ell)}\|\bm{p}\|_{L^\infty(\nu_\ell)}\big) \nonumber \\
&=~ \overbrace{\sum_{(\tau\iota\chi)} \Big( \sum_{(\rho\alpha\gamma)}^\gamma V_{,(\rho\alpha\gamma)(\tau\iota\chi)}(\bm{0})\big(D_{(\tau\iota\chi)}\bm{u}(\ell-\rho) - D_{(\tau\iota\chi)}\bm{u}(\ell)  \big) \Big)}^{A_1} \nonumber \\
&\qquad+~ \underbrace{\sum_{(\tau\iota\chi)} \Big(\sum_{(\rho\alpha\gamma)}^\gamma V_{,(\rho\alpha\gamma)(\tau\iota\chi)}(\bm{0}) -  \sum_{(\rho\gamma\alpha)}^\gamma V_{,(\rho\gamma\alpha)(\tau\iota\chi)}(\bm{0})\Big)D_{(\tau\iota\chi)}\bm{u}(\ell) }_{A_2} \nonumber \\
&\qquad\qquad+~ \mathcal{O}\big(\|\nabla^2 U\|_{L^\infty(\nu_\ell)}\|\nabla u\|_{L^\infty(\nu_\ell)} + \|\nabla \bm{p}\|_{L^\infty(\nu_\ell)}\|\bm{p}\|_{L^\infty(\nu_\ell)}\big) .
\end{align}
We then evaluate term $A_1$ and then Taylor expand the finite difference (keeping only the lowest order error terms) to produce
\begin{align}\label{taylor25}
A_1 
&=~ \sum_{(\tau\iota\chi)} \sum_{(\rho\alpha\gamma)}^\gamma V_{,(\rho\alpha\gamma)(\tau\iota\chi)}(\bm{0}) \Big( D_{-\rho} \big(U(\ell+\tau) - U(\eta) + p_\chi(\ell+\tau) - p_\iota(\ell) \big) \Big) \nonumber \\
&=~ \sum_{(\tau\iota\chi)} \sum_{(\rho\alpha\gamma)}^\gamma V_{,(\rho\alpha\gamma)(\tau\iota\chi)}(\bm{0}) \Big(D_{-\rho} \big(\nabla_\tau U(\ell) + p_\chi(\ell+\tau) - p_\chi(\ell) + p_\chi(\ell) - p_\iota(\ell) \big) \Big) \nonumber \\
&\qquad + \mathcal{O}\big(\|\nabla^3 U\|_{L^\infty(\nu_{\ell})} \big) \nonumber \\
&=~ \sum_{(\tau\iota\chi)} \sum_{(\rho\alpha\gamma)}^\gamma V_{,(\rho\alpha\gamma)(\tau\iota\chi)}(\bm{0}) \Big(D_{-\rho} \big(\nabla_\tau U(\ell) + \nabla_\tau p_\chi(\ell) + p_\chi(\ell) - p_\iota(\ell) \big) \Big) \nonumber \\
&\qquad + \mathcal{O}\big(\|\nabla^3 U\|_{L^\infty(\nu_{\ell})} + \|\nabla^2 \bm{p}\|_{L^\infty(\nu_{\ell})}\big) \nonumber \\
&=~ \sum_{(\tau\iota\chi)} \sum_{(\rho\alpha\gamma)}^\gamma V_{,(\rho\alpha\gamma)(\tau\iota\chi)}(\bm{0}) \Big( D_{-\rho} \big(\nabla_\tau U(\ell) + p_\chi(\eta) - p_\iota(\ell) \big) \Big) + \mathcal{O}\big(\|\nabla^3 U\|_{L^\infty(\nu_{\ell})} + \|\nabla^2 \bm{p}\|_{L^\infty(\nu_{\ell})}\big) \nonumber \\
&=~ \sum_{(\tau\iota\chi)} \sum_{(\rho\alpha\gamma)}^\gamma V_{,(\rho\alpha\gamma)(\tau\iota\chi)}(\bm{0}) \big( \nabla_{\tau,-\rho}^2 U(\ell) - \nabla_{\rho}p_\chi(\ell) + \nabla_{\rho} p_\iota(\ell) \big) + \mathcal{O}\big(\|\nabla^3 U\|_{L^\infty(\nu_{\ell})} + \|\nabla^2 \bm{p}\|_{L^\infty(\nu_{\ell})}\big).
\end{align}
As for the term $A_2$, we have
\begin{align}\label{taylor3}
&A_2 =~ \sum_{(\tau\iota\chi)} \Big(\sum_{(\rho\alpha\gamma)}^\gamma V_{,(\rho\alpha\gamma)(\tau\iota\chi)}(\bm{0}) -  \sum_{(\rho\gamma\alpha)}^\gamma V_{,(\rho\gamma\alpha)(\tau\iota\chi)}(\bm{0})\Big)D_{(\tau\iota\chi)}\bm{u}(\ell)  \nonumber \\
&=~ \sum_{(\tau\iota\chi)} \Big(\sum_{(\rho\alpha\gamma)}^\gamma V_{,(\rho\alpha\gamma)(\tau\iota\chi)}(\bm{0}) -  \sum_{(\rho\gamma\alpha)}^\gamma V_{,(\rho\gamma\alpha)(\tau\iota\chi)}(\bm{0})\Big)\big(U(\ell+\tau) - U(\ell) + p_\chi(\ell+\tau) - p_\iota(\ell)  \big) \nonumber  \\
&=~ \sum_{(\tau\iota\chi)} \Big(\sum_{(\rho\alpha\gamma)}^\gamma V_{,(\rho\alpha\gamma)(\tau\iota\chi)}(\bm{0}) -  \sum_{(\rho\gamma\alpha)}^\gamma V_{,(\rho\gamma\alpha)(\tau\iota\chi)}(\bm{0})\Big)\big( \nabla_\tau U(\ell) + \frac{1}{2}\nabla^2_{\tau}U(\ell) + p_\chi(\ell+\tau) - p_\chi(\ell) \nonumber  \\
&\qquad\qquad\qquad\qquad\qquad\qquad\qquad\qquad\qquad\qquad\qquad\qquad+ p_\chi(\ell) - p_\iota(\ell)  \big) + \mathcal{O}\big(\|\nabla^3 U\|_{L^\infty(\nu_{\ell})} \big)\nonumber \\
&=~ \sum_{(\tau\iota\chi)} \Big( \sum_{(\rho\alpha\gamma)}^\gamma V_{,(\rho\alpha\gamma)(\tau\iota\chi)}(\bm{0}) -  \sum_{(\rho\gamma\alpha)}^\gamma V_{,(\rho\gamma\alpha)(\tau\iota\chi)}(\bm{0})\Big) \big( \nabla_\tau U(\ell) + \frac{1}{2}\nabla^2_{\tau}U(\ell) + \nabla_\tau p_\chi(\ell)\nonumber  \\
&\qquad\qquad\qquad\qquad\qquad\qquad\qquad\qquad\qquad+ p_\chi(\ell) - p_\iota(\ell)  \big) + \mathcal{O}\big(\|\nabla^3 U\|_{L^\infty} + \|\nabla^2 \bm{p}\|_{L^\infty}\big) \nonumber \\
&=~ \overbrace{\sum_{(\tau\iota\chi)} \Big(\sum_{(\rho\alpha\gamma)}^\gamma V_{,(\rho\alpha\gamma)(\tau\iota\chi)}(\bm{0}) -  \sum_{(\rho\gamma\alpha)}^\gamma V_{,(\rho\gamma\alpha)(\tau\iota\chi)}(\bm{0})\Big)\big( \nabla_\tau U(\ell) + p_\chi(\ell) - p_\iota(\ell)  \big)}^{A_{21}} \nonumber \\
&\qquad +~ \underbrace{\sum_{(\tau\iota\chi)} \Big( \sum_{(\rho\alpha\gamma)}^\gamma V_{,(\rho\alpha\gamma)(\tau\iota\chi)}(\bm{0}) -  \sum_{(\rho\gamma\alpha)}^\gamma V_{,(\rho\gamma\alpha)(\tau\iota\chi)}(\bm{0})\Big) \big(\frac{1}{2}\nabla^2_{\tau}U(\ell) + \nabla_\tau p_\chi(\ell)  \big)}_{A_{22}} \nonumber \\
& \qquad\qquad+ \mathcal{O}\big(\|\nabla^3 U\|_{L^\infty(\nu_\ell)} + \|\nabla^2 \bm{p}\|_{L^\infty(\nu_\ell)}\big).
\end{align}
Inserting the expressions~\eqref{taylor25} and~\eqref{taylor3} into~\eqref{taylor2} yields the desired expression for $\frac{\partial \mathcal{E}^\a_{\rm hom}(\bm{u})}{\partial u_\gamma(\ell)}$ in~\eqref{at_force_eq}.

To establish \eqref{at_force_lin}, we observe that from Appendix \ref{app:cb_linear}, if $(U^0, \bm{p}^0)$ is a solution to the linear elastic Cauchy-Born model, then $(U^0, \bm{p}^0)$ also solves the mixed variational equation
\begin{align}\label{lin_cb_weak}
&\<\partial_{\mF \mF}W(0, \bm{0})\nabla U, \nabla V \> + \sum_{\nu} \< \partial_{\mF p_\nu}W(0, \bm{0}) \nabla U, q_\nu \> + \sum_{\mu}\< \partial_{ p_\mu \mF}W(0, \bm{0}) p_\mu, \nabla V \> \nonumber \\
&\qquad+ \sum_{\mu,\nu}\< \partial_{p_\mu p_\nu}W(0, \bm{0}) p_\mu, q_\nu \> = 0, \quad \forall (V,\bm{q}) \in C^\infty_0.
\end{align}
The strong form of \eqref{lin_cb_weak} reads
\begin{align}\label{lin_cb_strong}
\nabla \cdot \big(\partial_{\mF \mF}W(0, \bm{0})\nabla U\big) + \sum_{\mu} \nabla \cdot\big(\partial_{ p_\mu \mF}W(0, \bm{0}) p_\mu\big) =~& 0, \nonumber \\
\partial_{\mF p_\nu}W(0, \bm{0}) \nabla U + \sum_{\mu } \partial_{p_\mu p_\nu}W(0, \bm{0}) p_\mu =~& 0, \qquad \mbox{for each $\nu\in\mathcal{S}$.}
\end{align}
It is then a simple calculus exercise to compute (see also~\cite{olsonOrtner2016} for the same expressions)
\begin{align}\label{deriv_express}
\partial^2_{\mF_{mn} \bm{p}_\mu^l} W(0, \bm{0}) =~& \sum_{(\rho\alpha\mu)}^\mu \sum_{\tripleIota} V_{,(\rho\alpha\mu)\tripleIota}^{lm}(\bm{0})\tau_n - \sum_{(\rho\mu\alpha)}^\mu \sum_{\tripleIota} V_{,(\rho\mu\alpha)\tripleIota}^{lm}(\bm{0})\tau_n, \nonumber \\
\partial^2_{\mF_{mn} \mF_{rs}}W(0, \bm{0}) =~& \sum_{(\rho\alpha\mu)} \sum_{\tripleIota} V_{,(\rho\alpha\mu)\tripleIota}^{mr}(\bm{0})\rho_n\tau_s, \nonumber \\
\partial^2_{\bm{p}_\mu^k\bm{p}_\gamma^l} W(0, \bm{0}) =~& \sum_{(\tau\iota\mu) }^\mu \sum_{(\rho\alpha\gamma) }^\gamma V^{kl}_{,(\rho\alpha\gamma)(\tau\iota\mu) }( \bm{0} ) - \sum_{(\tau\iota\mu)}^\mu \sum_{(\rho\gamma\alpha)}^\gamma V^{kl}_{,(\rho\gamma\alpha)(\tau\iota\mu)}(\bm{0} ) \nonumber \\
&\quad +  \sum_{(\tau\mu\iota)}^\mu \sum_{(\rho\gamma\alpha)}^\gamma V^{kl}_{,(\rho\gamma\alpha)(\tau\mu\iota) }( \bm{0} ) - \sum_{(\tau\mu\iota)}^\mu \sum_{(\rho\alpha\gamma)}^\gamma V^{kl}_{,(\rho\alpha\gamma)(\tau\mu\iota)}(\bm{0}),
\end{align}
and therefore
\begin{align}\label{div1}
-\big[\nabla \cdot \big(\partial_{\mF\mF}W(0, \bm{0}) \nabla U\big)\big]_{m} =~& -\sum_{(\rho\alpha\mu)}\sum_{\tripleIota} V_{,(\rho\alpha\mu)\tripleIota}^{mr}(\bm{0}) \rho_n\tau_s u_{r,sn} \nonumber \\
=~& \sum_{(\rho\alpha\mu)}\sum_{\tripleIota} V_{,(\rho\alpha\mu)\tripleIota}^{mr}(\bm{0}) \nabla^2_{-\rho\tau}u_r,
\end{align}
and
\begin{align}\label{div2}
&\big[\nabla \cdot \big(\partial_{\mF p_\mu}W(0, \bm{0}) p_\mu\big)\big]_{m} \nonumber \\
&=~ \sum_{(\rho\alpha\mu) }\sum_{\tripleIota} V_{,(\rho\alpha\mu)\tripleIota}^{lm}(\bm{0})\tau_np_{\mu,n}^{l} - \sum_{(\rho\mu\alpha) }\sum_{\tripleIota} V_{,(\rho\mu\alpha)\tripleIota}^{lm}(\bm{0})\tau_np_{\mu,n}^{l} \nonumber \\
&=~ \sum_{(\rho\alpha\mu) }\sum_{\tripleIota} V_{,(\rho\alpha\mu)\tripleIota}^{lm}(\bm{0})(\nabla_\tau p_{\mu})^{l} - \sum_{(\rho\mu\alpha) }\sum_{\tripleIota} V_{,(\rho\mu\alpha)\tripleIota}^{lm}(\bm{0})(\nabla_\tau p_{\mu})^{l} \nonumber \\
&=~ \sum_{(\tau\iota\chi) }\sum_{(\rho\alpha\mu)} V_{,(\tau\iota\chi)(\rho\alpha\mu)}^{lm}(\bm{0})(\nabla_\rho p_{\chi})^{l} - \sum_{(\tau\chi\iota) }\sum_{(\rho\alpha\mu)} V_{,(\tau\chi\iota)(\rho\alpha\mu)}^{lm}(\bm{0})(\nabla_\rho p_{\chi})^{l}
\end{align}
where in obtaining the last line we have merely relabeled $\tau$ to $\rho$, $\iota$ to $\alpha$ and $\chi$ to $\mu$.
We use~\eqref{div1} and~\eqref{div2} to compute the summation over $\gamma$ of {\textbf{(1)}} in~\eqref{at_force_eq}, which is valid since $U^{\rm lin}$ is smooth in this region:
\begin{align}\label{sum1}
&\sum_{(\tau\iota\chi)} \sum_{(\rho\alpha\gamma)} V_{,(\rho\alpha\gamma)(\tau\iota\chi)}(\bm{0}) \big(\nabla_{\tau,-\rho}^2 U(\ell) - \nabla_{\rho}p_\chi(\ell) + \nabla_{\rho} p_\iota(\ell) \big) \nonumber \\
&= -\nabla \cdot \big(\partial_{\mF\mF}W(0, \bm{0}) \nabla U\big) + \sum_{(\tau\iota\chi)} \sum_{(\rho\alpha\gamma)} V_{,(\tau\iota\chi)(\rho\alpha\gamma)}(\bm{0})\big(- \nabla_{\rho}p_\chi(\ell) + \nabla_{\rho} p_\iota(\ell) \big)  \, \mbox{by Clairaut's Thm.} \nonumber \\
&= -\nabla \cdot \big(\partial_{\mF\mF}W(0, \bm{0}) \nabla U\big) - \sum_\mu \nabla \cdot \big(\partial_{\mF p_\mu}W(0, \bm{0}) p_\mu\big) = 0 \quad \mbox{by~\eqref{lin_cb_strong}}.
\end{align}
When summing over $\gamma$ in terms {\textbf{(2a)}} and {\textbf{(2b)}} of~\eqref{at_force_eq}, the terms in braces disappear, which when combined with~\eqref{sum1} yields the second component of~\eqref{at_force_lin}.  Moreover, even if $\gamma$ is not summed over, then term {\textbf{(2a)}} vanishes at the linear elastic solution due to the second constraint of~\eqref{lin_cb_strong} and the expressions~\eqref{deriv_express}. Indeed, we observe that
\begin{align}\label{eq:2a}
{\textbf{(2a)}} =& \sum_{(\tau\iota\chi)} \Big(\sum_{(\rho\alpha\gamma)}^\gamma V_{,(\rho\alpha\gamma)(\tau\iota\chi)}(\bm{0}) -  \sum_{(\rho\gamma\alpha)}^\gamma V_{,(\rho\gamma\alpha)(\tau\iota\chi)}(\bm{0})\Big)\big( \nabla_\tau U(\ell) + p_\chi(\ell) - p_\iota(\ell)  \big) \nonumber \\
=& \sum_{(\rho\alpha\gamma)}^\gamma\sum_{\tripleIota} V_{,(\rho\alpha\mu)\tripleIota}(\bm{0})\nabla_\tau U(\ell) - \sum_{(\rho\gamma\alpha)}^\gamma\sum_{\tripleIota} V_{,(\rho\gamma\alpha)\tripleIota}(\bm{0})\nabla_\tau U(\ell) \nonumber \\
&+ \sum_{(\tau\iota\chi)} \Big(\sum_{(\rho\alpha\gamma)}^\gamma V_{,(\rho\alpha\gamma)(\tau\iota\chi)}(\bm{0}) -  \sum_{(\rho\gamma\alpha)}^\gamma V_{,(\rho\gamma\alpha)(\tau\iota\chi)}(\bm{0})\Big)\big( p_\chi(\ell) - p_\iota(\ell)  \big) \nonumber \\
=&~ \partial_{p_\mu \mF} W(0, \bm{0}) + T_{\rm 2a},
\end{align}
where the last line follows from the identity
\begin{align}\label{nonDiv1}
\big[\partial_{p_\mu \mF} W(0, \bm{0}) \nabla U\big]_{l}= \sum_{(\rho\alpha\mu)}^\mu\sum_{\tripleIota} V_{,(\rho\alpha\mu)\tripleIota}^{lm}(\bm{0})(\nabla_\tau U)^m - \sum_{(\rho\mu\alpha)}^\mu\sum_{\tripleIota} V_{,(\rho\mu\alpha)\tripleIota}^{lm}(\bm{0})(\nabla_\tau U)^m.
\end{align}
As for the term $T_{\rm 2a}$, we have
\begin{align}\label{nonDiv2}
T_{\rm 2a} &=~ \sum_{(\tau\iota\chi)} \Big(\sum_{(\rho\alpha\gamma)}^\gamma V_{,(\rho\alpha\gamma)(\tau\iota\chi)}(\bm{0}) -  \sum_{(\rho\gamma\alpha)}^\gamma V_{,(\rho\gamma\alpha)(\tau\iota\chi)}(\bm{0})\Big)\big( p_\chi(\ell) - p_\iota(\ell)  \big) \nonumber\\
&=~ \sum_{(\tau\iota\chi)} \sum_{(\rho\alpha\gamma)}^\gamma V_{,(\rho\alpha\gamma)(\tau\iota\chi)}(\bm{0}) p_\chi(\ell) -  \sum_{(\tau\iota\chi)}\sum_{(\rho\gamma\alpha)}^\gamma V_{,(\rho\gamma\alpha)(\tau\iota\chi)}(\bm{0}) p_\chi(\ell)  \nonumber\\
&\quad-  \sum_{(\tau\iota\chi)} \sum_{(\rho\alpha\gamma)}^\gamma V_{,(\rho\alpha\gamma)(\tau\iota\chi)}(\bm{0}) p_\iota(\ell) +  \sum_{(\tau\iota\chi)}\sum_{(\rho\gamma\alpha)}^\gamma V_{,(\rho\gamma\alpha)(\tau\iota\chi)}(\bm{0}) p_\iota(\ell) \nonumber\\
&=~ \sum_{(\tau\iota\chi)} \sum_{(\rho\alpha\gamma)}^\gamma V_{,(\rho\alpha\gamma)(\tau\iota\chi)}(\bm{0}) p_\chi(\ell) -  \sum_{(\tau\iota\chi)}\sum_{(\rho\gamma\alpha)}^\gamma V_{,(\rho\gamma\alpha)(\tau\iota\chi)}(\bm{0}) p_\chi(\ell)  \nonumber\\
&\quad-  \sum_{(\tau\chi\iota)} \sum_{(\rho\alpha\gamma)}^\gamma V_{,(\rho\alpha\gamma)(\tau\chi\iota)}(\bm{0}) p_\chi(\ell) +  \sum_{(\tau\chi\iota)}\sum_{(\rho\gamma\alpha)}^\gamma V_{,(\rho\gamma\alpha)(\tau\chi\iota)}(\bm{0}) p_\iota(\ell)  \,\, \mbox{relabeling $\chi$ to $\iota$ and $\iota$ to $\chi$} \nonumber \\
&=~ \sum_\chi \partial_{\bm{p}_\chi \bm{p}_\gamma}W(0, \bm{0}) p_\chi.
\end{align}
Hence, we can obtain
\begin{align*}
{\textbf{(2a)}} = \partial_{p_\mu \mF} W(0, \bm{0}) + \sum_\chi \partial_{\bm{p}_\chi \bm{p}_\gamma}W(0, \bm{0}) p_\chi = 0 \quad \mbox{by~\eqref{lin_cb_strong}}.
\end{align*}
We may now replace the linear elastic solution with $(U^0, \bm{p}^0)$ and make an at most $|\ell|^{-3}$ error according to Lemma~\ref{diff_grad} in $\textbf{(1)}$ and $\textbf{(2b)}$ which yields~\eqref{at_force_lin} when combined with decay estimates for the remainder terms used previously in Lemma~\ref{diff_grad}.  Moreover, even though $S\bm{p}^0$ is not smooth, we are still able to use Taylor expansions of $\tilde{D}_\rho p^0$ just as we did in Lemma~\ref{diff_grad} so that we may replace $(U^0, \bm{p}^0)$ with $(SU^0, S\bm{p}^0)$ to obtain~\eqref{at_force_lin_s}. 
\end{proof}

\subsection{Proof of Theorem \ref{energy_thm}}
\label{sec:pf:energy}

We essentially adapt the proof of \cite[Lemma 3.2]{Ehrlacher2013} to the setting of multilattice.  Namely, we recall the definition of the atomistic energy as
\begin{align*}
\mathcal{E}^\a(\bm{u}) =~& \sum_{\ell \in \Lambda} V_{\ell}(D\bm{u}(\ell)),
\end{align*}
and note that if $\bm{u}$ belongs to the space $\bm{\mathcal{U}}_0$, then
\begin{align*}
\<\delta \mathcal{E}^\a(\bm{0}),\bm{u}\> = \frac{d}{dt}\big[ \mathcal{E}^\a(\bm{0}+ t\bm{u})\big]\Big|_{t = 0} = \sum_{\ell \in \Lambda} \< \delta V_\ell(\bm{0}), D\bm{u}(\ell) \>,
\end{align*}
as this becomes a finite sum (and hence we can differentiate term-by-term), and thus
\[
\mathcal{E}^\a(\bm{u}) = \sum_{\ell \in \Lambda}\Big[ V_{\ell}\big(D\bm{u}(\ell)\big) - \< \delta V_\ell(\bm{0}), D\bm{u}(\ell) \>\Big] + \<\delta \mathcal{E}^\a(\bm{0}),\bm{u}\>,
\]
for $\bm{u} \in \bm{\mathcal{U}}_0$. If we can show that (1) $ 
\sum_{\ell \in \Lambda}\big[ V_{\ell}(D\bm{u}(\ell)) - \< \delta V_\ell(\bm{0}), D\bm{u}(\ell) \>\big]$ is well-defined for displacements having finite energy and is differentiable and (2) $
\<\delta \mathcal{E}^\a(\bm{0}),\bm{u}\>$ is a bounded linear functional, then we will have that $\mathcal{E}^\a(\bm{u})$ agrees with a $C^k$ functional on the dense subset, $\bm{\mathcal{U}}_0$ of $\bm{\mathcal{U}}$, and hence may be uniquely extended to a $C^k$ functional on $\mathcal{A}$.

Showing that the function
\[
\sum_{\ell \in \Lambda}\Big[ V_{\ell}\big(D\bm{u}(\ell)\big) - \< \delta V_\ell(\bm{0}), D\bm{u}(\ell) \>\Big]
\]
is well defined and continuously differentiable can be done verbatim to~\cite[Theorem 2.1]{olsonOrtner2016} by simply replacing the homogeneous site potential, $V$, in that work with $V_\ell$ in the present work so we omit the details.

To show that $\<\delta \mathcal{E}^\a(\bm{0}),\bm{u}\> = \sum_{\ell \in \Lambda} \< \delta V_\ell(\bm{0}), D\bm{u}(\ell)\>$ is a bounded functional, we recall first the identity~\eqref{slip_fd_id}
\[
\tilde{D}\bm{u}(\ell) = \tilde{D}_\rho u_0 + \tilde{D}_\rho p_\beta(\ell) + p_\beta(\ell) - p_\alpha(\ell).
\]
Using this identity and the definition of $e(\ell)$ from~\eqref{ml_strain}, we have
\begin{align}
\sum_{\ell \in \Lambda} \< \delta V_\ell(\bm{0}), D\bm{u}(\ell)\> =~& \sum_{\ell \in \Lambda} \< \delta V(e(\ell)), \tilde{D}\bm{u}(\ell)\> \nonumber \\
=~&\sum_{\ell \in \Lambda}  \sum_{\triple}V_{, {\triple}} \big(e(\ell)\big)\cdot \tilde{D}_{\triple}\bm{u}(\ell) \nonumber \\
=~& \sum_{\ell \in \Lambda}  \sum_{\triple}V_{, {\triple}} \big(e(\ell)\big)\cdot\big(\tilde{D}_\rho u_0(\ell) + \tilde{D}_\rho p_\beta(\ell) + p_\beta(\ell) - p_\alpha(\ell)\big) \nonumber \\
=~& \overbrace{\sum_{\ell \in \Lambda}  \sum_{\triple}V_{, {\triple}} \big(e(\ell)\big)\cdot\big(\tilde{D}_\rho u_0(\ell) + \tilde{D}_\rho p_\beta(\ell) \big)}^{T_1} \nonumber \\
&~\qquad+ \underbrace{\sum_{\ell \in \Lambda}  \sum_{\triple}V_{, {\triple}} \big(e(\ell)\big)\cdot\big(p_\beta(\ell) - p_\alpha(\ell)\big)}_{T_2}.
\end{align}
For term $T_2$, we have
\begin{align}
T_2 =~&  \sum_{\ell\in\Lambda}\sum_{\triple}V_{, \triple} \big(e(\ell)\big)\cdot\big(p_\beta(\ell) - p_\alpha(\ell)\big) \nonumber \\
=~& \sum_{\ell\in\Lambda}\sum_{\triple}\Big(V_{, \triple} \big(e(\ell)\big) - V_{, {\triple}}(\bm{0}) + V_{, {\triple}}(\bm{0})\Big)\cdot\big( p_\beta(\ell) - p_\alpha(\ell)\big) \nonumber \\
=~& \sum_{\ell\in\Lambda}\sum_{\triple} \Big(V_{, \triple} \big(e(\ell)\big) - V_{\triple}(\bm{0})\Big)\cdot\big(p_\beta(\ell) - p_\alpha(\ell)\big) \nonumber \\
&\qquad + \sum_{\ell\in\Lambda}\sum_{\triple}\Big( V_{\triple}(\bm{0}) - V_{, (\rho\beta\alpha)}(\bm{0})\Big)\cdot p_\beta(\ell) \nonumber \\
=~& \sum_{\ell\in\Lambda}\sum_{\triple}\Big( V_{, \triple} \big(e(\ell)\big) - V_{, \triple}(\bm{0})\Big)\cdot\big(p_\beta(\ell) - p_\alpha(\ell)\big)  \qquad \mbox{by~\cite[Appendix A.4 ]{olsonOrtner2016}}\nonumber  \\
=~& \sum_{\ell\in\Lambda}\sum_{\triple,\tripleTau}\Big( V_{, \triple\tripleTau} \big(\bm{0}\big) e_{\tripleTau}(\ell) + \mathcal{O}(\|e_{\tripleTau}\|_{L^\infty}^2)\Big)\cdot\big(p_\beta(\ell) - p_\alpha(\ell)\big) \nonumber  \\
=~& \sum_{\ell\in\Lambda}\sum_{\tripleTau}\Big(\sum_{\triple}V_{\triple\tripleTau} \big(\bm{0}\big)- \sum_{(\rho\beta\alpha)}V_{(\rho\beta\alpha)\tripleTau} \big(\bm{0}\big) +\mathcal{O}(|\ell|^{-2})\Big)e_{\tripleTau}(\ell) \cdot p_\beta(\ell) \nonumber \\
=~& \sum_{\ell\in\Lambda}\sum_{\tripleTau}\Big(\sum_{\triple}V_{\triple\tripleTau} \big(\bm{0}\big)- \sum_{(\rho\beta\alpha)}V_{(\rho\beta\alpha)\tripleTau} \big(\bm{0}\big) +\mathcal{O}(|\ell|^{-2}) \Big)\big(\nabla_\tau U^0(\ell) + p_\delta(\ell) - p_\gamma(\ell) \nonumber  \\
&\hspace{3.7in}+ \mathcal{O}(|\ell|^{-2})\big)  \cdot p_\beta(\ell),
\end{align}
where we have used Lemma~\ref{diff_grad} in obtaining the final line.  Now recall that the term labeled as (2a) in Lemma~\ref{at_force} was shown to vanish in the proof of that Lemma. Moreover, this is exactly the term immediately above up to an  $\mathcal{O}(|\ell|^{-2})$ error. This further implies $T_2$ is summable so that
\[
|T_2| \lesssim~ \|\bm{p}\|_{\ell^2} \lesssim~ \|\bm{u}\|_{\a_1}.
\]
For term $T_1$, we utilize~\cite[Lemma 5.7]{Ehrlacher2013} which allows to write ``integrate $\tilde{D}$ by parts:''
\begin{equation}\label{t1}
\begin{split}
T_1 &=~\sum_{\ell \in \Lambda}  \sum_{\triple}V_{, \triple} \big(e(\ell)\big)\cdot\big(\tilde{D}_\rho u_0(\ell) + \tilde{D}_\rho p_\beta(\ell) \big) \\
&=~ \sum_{\ell \in \Lambda}  \sum_{\triple}\tilde{D}_{-\rho}\big(V_{, \triple} \big(e(\ell)\big)\big)\cdot\big( u_0(\ell) + p_\beta(\ell) \big) \\
&=~ \sum_{\ell \in \Lambda}  \sum_{\triple}\tilde{D}_{-\rho}\big(V_{, \triple} \big(e(\ell)\big)\big)\cdot u_0(\ell) + \sum_{\ell \in \Lambda}  \sum_{\triple}\tilde{D}_{-\rho}\big(V_{, \triple} \big(e(\ell)\big)\big)\cdot p_\beta(\ell) . 
\end{split}
\end{equation}
Next, observe that by Lemma~\ref{at_force}, if $\ell \notin \Omega_\Gamma$,
\begin{equation}\label{force_elastic1}
\Big|\sum_{\triple}\tilde{D}_{-\rho}\big(V_{, \triple} \big(e(\ell)\big)\big)\Big|  = \Big|\sum_{\triple}D_{-\rho}\big(V_{, \triple} \big(D\bm{u}^0\big)\big)\Big| = \Big|\sum_{\gamma}\frac{\partial \mathcal{E}^\a_{\rm hom}(\bm{u})}{\partial u_\gamma(\ell)}\Big|_{(U^0,\bm{p}^0)}\Big| \lesssim~ |\ell|^{-3} ,
\end{equation}
and if $\ell \in \Omega_\Gamma$,
\begin{equation}\label{force_elastic2}
\begin{split}
\Big|\sum_{\triple}\tilde{D}_{-\rho}\big(V_{, \triple} \big(e(\ell)\big)\big)\Big|  =~& \Big|\sum_{\triple}RD_{-\rho}S\big(V_{, \triple} \big(RDS_0\bm{u}^0\big)\big)\Big| \nonumber  \\
=~& \Big|\sum_{\triple}RD_{-\rho}SR\big(V_{, \triple} \big(DS_0\bm{u}^0\big)\big)\Big| \nonumber  \\
=~& \Big|\sum_{\triple}RD_{-\rho}\big(V_{, \triple} \big(DS_0\bm{u}^0\big)\big)\Big| \nonumber  \\
=~& \Big|\sum_{\gamma}\frac{\partial \mathcal{E}^\a_{\rm hom}(\bm{u})}{\partial u_\gamma(\ell)}\Big|_{(S_0U^0,S\bm{p}^0)}\Big| \lesssim~ |\ell|^{-3}.
\end{split}
\end{equation}

According to \eqref{second_diff} and \eqref{third_diff}, we have 
\begin{align}\label{force_elastic3}
&\Big|\tilde{D}_{-\rho}\big(V_{, \triple} \big(e(\ell)\big)\big)\Big| \nonumber \\
\leq~& \Big|\tilde{D}_{-\rho}\big(V_{, \triple} (\bm{0})\big)\Big| + \Big|\tilde{D}_{-\rho}\big(\sum_{\tripleTau} V_{, \triple\tripleTau}(\bm{0}) e_{\tripleTau}(\ell)\big)\Big| + \mathcal{O}(|e_{\tripleTau}(\ell)|^2) \nonumber \\[-0.2cm]
\lesssim~& |\ell|^{-2}.
\end{align}

Inserting estimates~\eqref{force_elastic1},~\eqref{force_elastic2}, and~\eqref{force_elastic3} into~\eqref{t1}, we obtain
\begin{equation}\label{t12}
|T_1| \lesssim~  \Big|\sum_{\ell \in \Lambda} f(\ell)\cdot u_0(\ell) \Big|+ \sum_{\ell \in \Lambda}\sum_{\triple}\big|1+|\ell|\big|^{-2}|p_\beta(\ell)|,
\end{equation}
where $|f(\ell)| \lesssim~ |\ell|^{-3}$ for sufficiently large $|\ell|$.  We may then apply~\cite[Corollary 5.2]{Ehrlacher2013} to deduce the existence of $\bm{g}: \Lambda \to (\mathbb{R}^3)^{\mathcal{R}_1}$ such that
\[
\sum_{\ell \in \Lambda} f(\ell)\cdot u_0(\ell) = \sum_{\ell \in \Lambda} \sum_{\rho \in \mathcal{R}_1} g_\rho(\ell) D_\rho u_0(\ell)
\]
where $|g_\rho(\ell)| \lesssim |\ell|^{-2}$ for sufficiently large $|\ell|$.  This can then be inserted into~\eqref{t12} to yield
\[
|T_1| \lesssim \sum_{\ell \in \Lambda}\sum_{\beta \in \mathcal{S}} \big|1+|\ell|\big|^{-2}\big|D_\rho u_0(\ell) + p_\beta(\ell)\big| \lesssim~ \|\bm{u}\|_{\a_1},
\]
where we have applied the Cauchy-Schwarz inequality and summability of $|\ell|^{-2}$ in the final inequality. Combining our estimates for $T_1$ and $T_2$ shows $\<\delta \mathcal{E}^\a(\bm{0}),\cdot\>$ is a bounded linear functional and thus completes the proof.

\section{Proofs: Regularity}
\label{sec:proof:decay_thm}
In this section we prove the regularity result, Theorem \ref{decay_thm}.
The main idea of proving Theorem~\ref{decay_thm} is to show that $\bm{u}^\infty$ solves a linearized problem whose Green's function may be estimated in terms of existing Green's function estimates developed in~\cite{olsonOrtner2016} for point defects in multilattices.  The residual terms found in this linearization process are estimated in close analogy to~\cite{Ehrlacher2013}, and then the two estimates are combined to yield the theorem.  It is with this breakdown in mind that we split this section into separate subsections.  In Section~\ref{lin_res}, we derive the linearized problem and corresponding estimates on the residual; in Section~\ref{greeny} we recall the needed properties of the Green's function (matrix). Finally, in Section~\ref{pure_analysis}, we combine these results in a ``pure'' analysis problem to derive the estimates~\eqref{decay1}.



\subsection{Linearized Problem and Residual Estimates}\label{lin_res}

Our goal here is to establish the linearized problem and give the corresponding residual estimates. The linearized equation that $\bm{u}^\infty$ satisfies is formed by linearizing the defect-free energy $\mathcal{E}^{\rm a}_{\rm hom}(\bm{u})$ about the reference state $\bm{u}=\bm{0}$. The key point in this linearization is that the residual is quadratic in terms of the defect solution.

\begin{lemma}\label{lin_lemma}
Let $\bm{u}^\infty \in \mathcal{A}$ be a local minimizer of $\mathcal{E}^\a(\bm{u})$, there exists $\bar{f}: \Lambda \to \big(\mathbb{R}^3\big)^{\mathcal{R}}$ such that
\begin{align*}
\sum_{\ell \in \Lambda} \<\delta^2 V(\bm{0}) \tilde{D}\bm{u}^\infty, \tilde{D}\bm{v}\> =~& \sum_{\ell \in \Lambda} \sum_{\triple \in \mathcal{R}} \bar{f}_{\triple}(\ell)\cdot \tilde{D}_{\triple}\bm{v}(\ell) - \<\delta \mathcal{E}^\a(\bm{0}),\bm{v}\>, \quad \forall v \in \bm{\mathcal{U}}_0,
\end{align*}
where $\bar{f}_{\triple}$ satisfies
\begin{align*}
|\bar{f}_{\triple}(\ell)| \lesssim~ |\ell|^{-2} + |\tilde{D}\bm{u}^\infty(\ell)|^{2}_{\mathcal{R}}, \quad \forall (\rho\alpha\beta) \in \mathcal{R}. 
\end{align*}
\end{lemma}

\begin{proof}
Recall $\bm{u}^\infty$ solves the atomistic Euler-Lagrange equations,
\[
\<\delta \mathcal{E}^\a(\bm{u}^\infty),\bm{v} \> = 0, \quad \forall \bm{v} \in \bm{\mathcal{U}}_0.
\]
Using \eqref{slip_strain}, we may simply expand this as
\begin{align*}
0 =~& \sum_{\ell \in \Lambda} \<\delta V(e(\ell) + \tilde{D}\bm{u}^\infty(\ell)), \tilde{D}\bm{v}\> \\
=~& \sum_{\ell \in \Lambda} \<\delta V(e(\ell) + \tilde{D}\bm{u}^\infty(\ell)) - \delta V(e(\ell)) - \delta^2 V(e(\ell))\tilde{D}\bm{u}^\infty(\ell), \tilde{D}\bm{v}\>  \\
&+~ \sum_{\ell \in \Lambda} \<\big(\delta^2 V(e(\ell)) - \delta^2 V(\bm{0})\big)\tilde{D}\bm{u}^\infty(\ell), \tilde{D}\bm{v}\> \\
&+~ \sum_{\ell \in \Lambda} \<\delta^2 V(\bm{0})\tilde{D}\bm{u}^\infty(\ell), \tilde{D}\bm{v}\> + \sum_{\ell \in \Lambda} \<\delta V(e(\ell)), \tilde{D}\bm{v}\>.
\end{align*}

We then rearrange the terms to arrive at
\begin{align}\label{rearr}
\sum_{\ell \in \Lambda} \<\delta^2 V(\bm{0})\tilde{D}\bm{u}^\infty(\ell), \tilde{D}\bm{v}\> =~& -\sum_{\ell \in \Lambda} \<\delta V(e(\ell) + \tilde{D}\bm{u}^\infty(\ell)) - \delta V(e(\ell)) - \delta^2 V(e(\ell))\tilde{D}\bm{u}^\infty(\ell), \tilde{D}\bm{v}\>  \nonumber \\
&- \sum_{\ell \in \Lambda} \<\big(\delta^2 V(e(\ell)) - \delta^2 V(\bm{0})\big)\tilde{D}\bm{u}^\infty(\ell), \tilde{D}\bm{v}\>  \nonumber \\
&- \<\delta \mathcal{E}^\a(\bm{0}),\bm{v}\>, 
\end{align}
where in the last line we use the identity $\<\delta \mathcal{E}^\a(\bm{0}),\bm{v}\> = \sum_{\ell \in \Lambda} \<\delta V(e(\ell)), \tilde{D}\bm{v}\>$. Given $(\rho\alpha\beta)\in\mathcal{R}$, upon defining
\begin{align*}
\bar{f}_{\triple}(\ell) =~& -V_{, \triple}\big(e(\ell) +  \tilde{D}\bm{u}^\infty(\ell)\big) + V_{, \triple}(e(\ell)) + \sum_{\tripleTau} V_{, \triple\tripleTau}(e(\ell))\tilde{D}_{\tripleTau}\bm{u}^\infty(\ell) \\
&- \sum_{\tripleTau} \big(V_{, \triple\tripleTau}(e(\ell))\tilde{D}_{\tripleTau}\bm{u}^\infty(\ell) +  V_{, \triple\tripleTau}(\bm{0})\tilde{D}_{\tripleTau}\bm{u}^\infty(\ell) \big),
\end{align*}
it only remains to establish the given estimate on $\bar{f}_{\triple}(\ell)$. However, this is a straightforward consequence of Taylor's expansions and the aforementioned decay estimates on $e_{\tripleTau}(\ell)$ stated in Lemma~\ref{diff_grad}.

\end{proof}

Next, we prove the estimate of the term $\<\delta \mathcal{E}^\a(\bm{0}),\bm{v}\> = \sum_{\ell \in \Lambda} \<\delta V(e(\ell)), \tilde{D}\bm{v}\>$, and then give the residual estimate of the linearized problem in the following theorem.

\begin{theorem}\label{lin_prob_thm}
Let $\bm{u}^\infty \in \mathcal{A}$ be a local minimizer of $\mathcal{E}^\a(\bm{u})$, there exists $\bar{f}: \Lambda \to \big(\mathbb{R}^3\big)^{\mathcal{R}}$, $\bm{g}: \Lambda \to \big(\mathbb{R}^3\big)^{\mathcal{R}_1}$  and $\bm{k}:\Lambda \to \big(\mathbb{R}^3\big)^{\mathcal{S}}$ such that for $\bm{v} = (v_0, \bm{q}) \in \bm{\mathcal{U}}_0$,
\begin{align}
\<\tilde{H}\bm{u}^\infty,\bm{v}\> :=& \sum_{\ell \in \Lambda} \<\delta^2 V(\bm{0}) \tilde{D}\bm{u}^\infty, \tilde{D}\bm{v}\> \nonumber \\
=& \sum_{\ell \in \Lambda}\big(  \sum_{\triple \in \mathcal{R}}\bar{f}_{\triple}(\ell)\cdot \tilde{D}_{\triple}\bm{v}(\ell) + \<\bm{g}(\ell), Dv_0(\ell)\> + \<\bm{k}(\ell),\bm{q}(\ell)\>\big) ,
\end{align}
where for sufficiently large $|\ell|$, $f_{\triple}(\ell), g_{\rho}(\ell), k_\gamma(\ell)$ satisfy
\begin{align*}
|\bar{f}_{\triple}(\ell)| \lesssim~& |\ell|^{-2} + |\tilde{D}\bm{u}^\infty(\ell)|^{2}_{\mathcal{R}}, \quad\forall (\rho\alpha\beta)\in \mathcal{R}, \\
|g_{\rho}(\ell)| \lesssim~& |\ell|^{-2}, \quad \forall \rho \in \mathcal{R}_1, \\
|k_{\gamma}(\ell)| \lesssim~&  |\ell|^{-2}, \quad \forall \gamma \in \mathcal{S}.
\end{align*}
\end{theorem}

\begin{proof}
From Lemma~\ref{lin_lemma},
\begin{align*}
\sum_{\ell \in \Lambda} \<\delta^2 V(\bm{0}) \tilde{D}\bm{u}^\infty, \tilde{D}\bm{v}\> =~& \sum_{\ell \in \Lambda} \sum_{\triple \in \mathcal{R}} \bar{f}_{\triple}(\ell)\cdot \tilde{D}_{\triple}\bm{v}(\ell) -\sum_{\ell \in \Lambda} \<\delta V(\tilde{D}\bm{u}^{0}(\ell)), \tilde{D}\bm{v}\>,
\end{align*}
and
\begin{align}\label{e1}
\sum_{\ell \in \Lambda} \<\delta V(\tilde{D}\bm{u}^{0}(\ell)), \tilde{D}\bm{v}\> &= \sum_{\ell \in \Lambda} \sum_{\tripleR} V_{, \triple}(\tilde{D}\bm{u}^{0}(\ell))\cdot \tilde{D}_{\triple} \bm{v}(\ell) \nonumber \\
&=~ \sum_{\ell \in \Lambda} \sum_{\tripleR} V_{, \triple}(\tilde{D}\bm{u}^{0}(\ell))\cdot \tilde{D}_\rho v_0(\ell) \nonumber \\
 &\quad+ \sum_{\ell \in \Lambda} \sum_{\tripleR} V_{, \triple}(\tilde{D}\bm{u}^{0}(\ell))\cdot \tilde{D}_\rho q_\beta(\ell) \nonumber \\
&\quad+ \sum_{\ell \in \Lambda} \sum_{\tripleR} V_{, \triple}(\tilde{D}\bm{u}^{0}(\ell))\cdot \big( q_\beta(\ell) - q_\alpha(\ell)\big) \nonumber \\
&=:~ B_1 + B_2 + B_3.
\end{align}
We observe that $B_3$ was exactly the term, $T_2$, estimated in Theorem~\ref{energy_thm}, where we saw that
\[
B_3 = \<\bm{k}_1,\bm{q}\>, \qquad |k_1(\ell)| \lesssim |\ell|^{-2}.
\]
Meanwhile, $B_2$ is the second term of $T_1$ in~\eqref{t1}, which we saw could be estimated by
\begin{align*}
B_2 =~& \sum_{\ell \in \Lambda} \sum_{\tripleR} V_{, \triple}(\tilde{D}\bm{u}^{0}(\ell))\cdot \tilde{D}_\rho q_\beta(\ell) \\
=~&\sum_{\ell \in \Lambda} \tilde{D}_{-\rho}\sum_{\tripleR} V_{, \triple}(e(\ell))\cdot q_\beta(\ell) = \<\bm{k}_2,\bm{q}\>, \qquad |k_2(\ell)| \lesssim |\ell|^{-2}.
\end{align*}
Thus,
\begin{equation}\label{hequation}
B_2 + B_3 = \<\bm{k}, \bm{q}\>, \qquad |k(\ell)| \lesssim |\ell|^{-2}.
\end{equation}

Next, we observe $B_1$ is precisely the first term estimated in $T_1$ in~\eqref{t1} so, as in Theorem~\ref{energy_thm}, there exists $\bm{g}: \Lambda \to (\mathcal{R}^d)^{\mathcal{R}_1}$ with $|g(\ell)| \lesssim~ |\ell|^{-2}$ and
\begin{align*}
B_1 =~& 
\sum_{\ell \in \Lambda} \sum_{\rho \in \mathcal{R}_1} g_\rho (\ell)\cdot D_\rho v_0(\ell).
\end{align*}

\end{proof}

\subsection{Green's Function}\label{greeny}

Having defined our linearized problem and estimated the residual in Theorem~\ref{lin_prob_thm}, we will proceed to estimate the decay of $\bm{u}^\infty$ in the current section.  Doing that will require comparing the atomistic Green's matrix for the homogeneous energy derived in~\cite{olsonOrtner2016} for point defects to the Green's matrix for the dislocation solution.  Thus, we will introduce the homogeneous Green's function and the corresponding estimates for its decay here.

From~\cite{olsonOrtner2016}, we let $\xi \in \mathcal{B}$, where $\mathcal{B}$ is the first Brillouin zone associated to the atomic lattice, denote the Fourier variable and set
\begin{align*}
\hat{H}_{00}(\xi) :=~& \sum_{\triple\in \mathcal{R}}\sum_{\tripleTau\in \mathcal{R}} (e^{-2\pi i \xi \cdot \tau}-1)V_{,\triple\tripleTau}(\bm{0})(e^{2\pi i \xi \cdot \rho}-1), \\
[\hat{H}_{0\bm{p}}(\xi)]_\beta :=~&  \sum_{\rho\in\mathcal{R}_1}\sum_{\alpha = 0}^{S-1}\sum_{\tripleTau\in \mathcal{R}} \big[(e^{-2\pi i \xi\cdot \tau}-1)V_{,(\rho\alpha\beta)(\tau\gamma\delta)}(\bm{0})(e^{2\pi i \xi\cdot \rho}) \\
&\hspace{3in}- (e^{-2\pi i \xi\cdot \tau}-1)V_{,(\rho\beta\alpha)(\tau\gamma\delta)}(\bm{0})\big],  \\
[\hat{H}_{\bm{p}0}(\xi)]_\delta :=~& \sum_{\tau\in\mathcal{R}_1}\sum_{\gamma = 0}^{S-1}\sum_{\triple\in \mathcal{R}} \big[(e^{-2\pi i \xi\cdot \tau}) V_{,(\rho\alpha\beta)(\tau\gamma\delta)}(\bm{0})(e^{2\pi i \xi\cdot \rho} - 1) \\
&\hspace{3in}-   V_{,(\rho\alpha\beta)(\tau\delta\gamma)}(\bm{0})(e^{2\pi i \xi\cdot \rho} - 1)\big], \\
[\hat{H}_{\bm{p}\bm{p}}(\xi)]_{\beta\delta} :=~& \sum_{\rho,\tau \in \mathcal{R}_1}\sum_{\alpha,\gamma = 0}^{S-1}\left[e^{-2\pi i \xi\cdot \tau} V_{,(\rho\alpha\beta)(\tau\gamma\delta)}(\bm{0}) e^{2\pi i \xi\cdot \rho} + V_{,(\rho\beta\alpha)(\tau\delta\gamma)}(\bm{0}) - e^{-2\pi i \xi\cdot \tau}V_{,(\rho\beta\alpha)(\tau\gamma\delta)}(\bm{0}) \right. \\
&\hspace{3in}\left. - e^{2\pi i \xi \cdot \rho} V_{,(\rho\alpha\beta)(\tau\delta\gamma)}(\bm{0})\right],
\end{align*}
and define the {\em dynamical matrix}~\cite{wallace1998},
\begin{equation}\label{dynam_matrix}
   \hat{H}(\xi) := \begin{bmatrix} \hat{H}_{00}(\xi) & \hat{H}_{0\bm{p}}(\xi) \\ \hat{H}_{\bm{p}0}(\xi) & \hat{H}_{\bm{p}\bm{p}}(\xi)\end{bmatrix}.
\end{equation}

It has the property that
\begin{equation}\label{planch}
\<H\bm{u}^\infty, \bm{v} \> := \<\delta^2 \mathcal{E}^\a_{\rm hom}(\bm{0})\bm{u}^\infty, \bm{v}\> =~ \int_{\mathcal{B}} \begin{bmatrix} \hat{Z}(\xi) \\ \hat{\bm{q}}(\xi)\end{bmatrix}^* \hat{H}(\xi) \begin{bmatrix} \hat{U}^\infty(\xi) \\ \hat{\bm{p}}^\infty(\xi) \end{bmatrix}\, d\xi, \quad \forall \bm{v} = (Z,\bm{q}) \in \bm{\mathcal{U}}.
\end{equation}
The atomistic Green's matrix for the homogeneous energy is then defined by
\begin{equation}\label{green_mat}
\mathcal{G} := (\hat{H}^{-1})^{\vee},
\end{equation}
where $\vee$ denotes the inverse Fourier transform. It follows from Assumption~\ref{coercive} that $H$ is an invertible operator (see also~\cite{olsonOrtner2016}).  Upon partitioning $\mathcal{G}$ as
\[
\begin{pmatrix}
G_{00}      &G_{0\bm{p}}\\
G_{\bm{p}0} &G_{\bm{p}\bm{p}}
\end{pmatrix},
\]
several important decay estimates for the individual blocks were established in~\cite[Theorem 4.4]{olsonOrtner2016}:     for $\bm{\rho} \in (\mathcal{R}_1)^t$, $t \geq 0$ and
   $|\bm{\rho}| := t \in \mathbb{Z}$,
\begin{align}
\left|D_{\bm{\rho}}G_{00}(\ell)\right| \lesssim~& (1+|\ell|)^{-d-|\bm{\rho}|+2},  \qquad |\bm{\rho}| \geq 1, \label{hessianDecay1} \\
\left|D_{\bm{\rho}}G_{0\bm{p}}(\ell)\right| \lesssim~& (1+|\ell|)^{-d-|\bm{\rho}|+1},  \qquad |\bm{\rho}| \geq 0, \label{hessianDecay2}\\
\left|D_{\bm{\rho}}G_{\bm{p}\bm{p}}(\ell)\right| \lesssim~& (1+|\ell|)^{-d-|\bm{\rho}|},  \qquad \quad  |\bm{\rho}| \geq 0. \label{hessianDecay3}
\end{align}

\subsection{Analysis Problem}\label{pure_analysis}

We fix $\ell \in \Lambda$ and define the lattice function $\bm{v} = (Z,\bm{q})$ by $\bm{v} = \sum_{m}\mathcal{G}(k-\ell)\bm{e}_m$ where $\hat{e}_m$ is the $m$th standard basis vector. Using the relation~\eqref{planch}, we may write our infinite-domain dislocation solution as
\begin{equation}\label{conv}
\begin{pmatrix}
U^\infty(\ell) \\
\bm{p}^\infty(\ell)
\end{pmatrix}
=
\<H\bm{u}^\infty, \bm{v}\> = \sum_m\int_{\mathcal{B}} e^{2\pi i \xi \cdot \ell} \begin{bmatrix} \hat{U}^\infty(\xi) \\ \hat{\bm{p}}^\infty(\xi) \end{bmatrix}_m \hat{e}_m\, d\xi,
\end{equation}
where $H$ is defined in~\eqref{planch} and $\mathcal{G}$ is the associated Green's matrix defined in~\eqref{green_mat}.
Thus,
\begin{equation}\label{conv_var}
\begin{split}
\begin{pmatrix}
D_\tau U^\infty(\ell) \\
\bm{p}^\infty(\ell)
\end{pmatrix}
=~& \sum_m\left<H(k)\begin{pmatrix}
U^\infty(k) \\
\bm{p}^\infty(k)
\end{pmatrix},\begin{pmatrix}
D_\tau G_{00}(k-\ell)      &G_{0\bm{p}}(k-\ell)\\
D_\tau G_{\bm{p}0}(k-\ell) &G_{\bm{p}\bm{p}}(k-\ell)
\end{pmatrix} \hat{e}_m \right>.
\end{split}
\end{equation}
As shown in~\cite{Ehrlacher2013}, we consider two cases depending on the location of $\ell$ in the lattice:
\vskip-2cm
\subsubsection{Case 1} $B_{3/4|\ell|}(\ell) \cap \Gamma = \emptyset$

In this case we take a bump function $\eta(x)$ from~\cite{Ehrlacher2013}: define $s_1 := 1/2|\ell|-r_{\rm cut}, s_2 := 1/2|\ell|, \eta = 1$ in $B_{s_1/2}(\ell)$, $\eta = 0$ outside of $B_{s_1}(\ell)$, and $|\nabla \eta(x)| \lesssim~ |\ell|^{-1}$. We then make the substitution
\begin{equation}\label{eta_sub}
\begin{split}
\begin{pmatrix}
D_\tau U^\infty(\ell) \\
\bm{p}^\infty(\ell)
\end{pmatrix}
=~& \sum_m\left< H(k)\begin{pmatrix}
U^\infty(k) \\
\bm{p}^\infty(k)
\end{pmatrix},
\eta(k-\ell)\begin{pmatrix}
D_\tau G_{00}(k-\ell)      &G_{0\bm{p}}(k-\ell)\\
D_\tau G_{\bm{p}0}(k-\ell) &G_{\bm{p}\bm{p}}(k-\ell)
\end{pmatrix}\hat{e}_m \right> \\
&~+ \left<H(k)\begin{pmatrix}
U^\infty(k) \\
\bm{p}^\infty(k)
\end{pmatrix},(1-\eta(k-\ell))\begin{pmatrix}
D_\tau G_{00}(k-\ell)      &D_\tau G_{0\bm{p}}(k-\ell)\\
G_{\bm{p}0}(k-\ell) &G_{\bm{p}\bm{p}}(k-\ell)
\end{pmatrix}\hat{e}_m \right>
\end{split}
\end{equation}
Where $\eta \neq 0$, $D = \tilde{D}$ in this case so that from Theorem~\ref{lin_prob_thm}
\begin{equation}
\begin{split}
&\sum_m\left< H(k)\begin{pmatrix}
U^\infty(k) \\
\bm{p}^\infty(k)
\end{pmatrix},
\eta(k-\ell)\begin{pmatrix}
D_\tau G_{00}(k-\ell)      &G_{0\bm{p}}(k-\ell)\\
D_\tau G_{\bm{p}0}(k-\ell) &G_{\bm{p}\bm{p}}(k-\ell)
\end{pmatrix}\hat{e}_m \right>  \\
&=
\sum_m\left< \tilde{H}(k)\begin{pmatrix}
U^\infty(k) \\
\bm{p}^\infty(k)
\end{pmatrix},
\eta(k-\ell)\begin{pmatrix}
D_\tau G_{00}(k-\ell)      &G_{0\bm{p}}(k-\ell)\\
D_\tau G_{\bm{p}0}(k-\ell) &G_{\bm{p}\bm{p}}(k-\ell)
\end{pmatrix}\hat{e}_m \right> \\
&= \sum_m\sum_{\ell \in \Lambda}\sum_{\triple\in\mathcal{R}} D_{\triple}\left(\eta(k-\ell)
\begin{pmatrix}
D_\tau G_{00}(k-\ell) & G_{0\bm{p}}(k-\ell)\\
D_\tau G_{\bm{p}0}(k-\ell)   &G_{\bm{p}\bm{p}}(k-\ell)
\end{pmatrix}\hat{e}_m \right)\cdot \bar{f}_{\triple}(k) \\
&~+~ \sum_m\sum_{\ell \in \Lambda}\sum_{\rho\in\mathcal{R}_1} D_{\rho}\Big(\eta(k-\ell)
\begin{pmatrix}
D_\tau G_{00}(k-\ell) &
G_{0\bm{p}}(k-\ell)
\end{pmatrix}_m\Big)\cdot g_\rho(k) \quad \mbox{$(\cdot)_m$ is $m$th column} \\
&~+~\sum_{m}\sum_{\ell \in \Lambda} \eta(k-\ell)
\begin{pmatrix}
D_\tau G_{0\bm{p}}(k-\ell) &
G_{\bm{p}\bm{p}}(k-\ell)
\end{pmatrix}_m \cdot \bm{k}(k).
\end{split}
\end{equation}
We may now utilize the decay estimates for $\mathcal{G}$ in~\eqref{hessianDecay1},~\eqref{hessianDecay2},~\eqref{hessianDecay3}, for $\eta$, and those for $\bar{f},g,\bm{k}$ in Theorem~\ref{lin_prob_thm} to obtain
\begin{equation}
\begin{split}
&\sum_m\left<H(k)\begin{pmatrix}
U^\infty(k) \\
\bm{p}^\infty(k)
\end{pmatrix},\eta(k-\ell)\begin{pmatrix}
D_\tau G_{00}(k-\ell)      & G_{0\bm{p}}(k-\ell)\\
D_\tau G_{\bm{p}0}(k-\ell) &G_{\bm{p}\bm{p}}(k-\ell)
\end{pmatrix}\hat{e}_m \right>  \\
&\lesssim \sum_{k \in B_{s_2}(\ell)} \left( |k|^{-2} + |\tilde{D}\bm{u}^\infty(k)|^2 \right)|\ell-k|^{-2} \\
&=~  \sum_{k \in B_{s_2}(\ell)} \left( |k|^{-2} + |D\bm{u}^\infty(k)|^2 \right)|\ell-k|^{-2}.
\end{split}
\end{equation}

Next, we estimate
\begin{equation}
\begin{split}
&\left<H(k)\begin{pmatrix}
U^\infty(k) \\
\bm{p}^\infty(k)
\end{pmatrix},(1-\eta(k-\ell))\begin{pmatrix}
D_\tau G_{00}(k-\ell)      & G_{0\bm{p}}(k-\ell)\\
D_\tau G_{\bm{p}0}(k-\ell) &G_{\bm{p}\bm{p}}(k-\ell)
\end{pmatrix}\hat{e}_m  \right>\\
&=~ \sum_{k \in \ell}\<\delta^2 V(0)D\bm{u}^\infty,D\Big((1-\eta(k-\ell))\begin{pmatrix}
D_\tau G_{00}(k-\ell)      & G_{0\bm{p}}(k-\ell)\\
D_\tau G_{\bm{p}0}(k-\ell) &G_{\bm{p}\bm{p}}(k-\ell)
\end{pmatrix}\hat{e}_m \Big)\>.
\end{split}
\end{equation}
Using the product rule for finite differences, decay estimates on $\mathcal{G}$ and $|\nabla \eta(x)| \lesssim |x|^{-1}$, it is then straightforward to estimate
\begin{equation}
\begin{split}
&\left<H(k)\begin{pmatrix}
U^\infty(k) \\
\bm{p}^\infty(k)
\end{pmatrix},(1-\eta(k-\ell))\begin{pmatrix}
D_\tau G_{00}(k-\ell)      & G_{0\bm{p}}(k-\ell)\\
D_\tau G_{\bm{p}0}(k-\ell) &G_{\bm{p}\bm{p}}(k-\ell)
\end{pmatrix}\hat{e}_m  \right> \\
&\lesssim~ \sum_{k \in \Lambda, |k-\ell| \gtrsim s_1/4 } |D\bm{u}(k)|^\infty|k-\ell|^{-2}.
\end{split}
\end{equation}

We now remark that these estimates are set up to be identical to those used to obtain~\cite[Equation 6.31]{Ehrlacher2013}, and thus we likewise obtain for sufficiently large $\ell$
\begin{equation}\label{641}
\max\{|DU^{\infty}(\ell)|, |\bm{p}^\infty(\ell)|\} \lesssim~ |\ell|^{-1} + \|(1+|\ell-k|)^{-1}D\bm{u}^\infty(k)\|_{\ell^2(\Lambda \cap B_{s_2}(\ell))}.
\end{equation}

\subsubsection{Case 2} $B_{3/4|\ell|}(\ell) \cap \Gamma \neq \emptyset$

We argue as in~\cite{Ehrlacher2013} by using a reflection argument whereby the branch cut $\Gamma = \{(x_1, x_2) : x_2 = \hat{x}_2, x_1 > \hat{x}_1\}$ is replaced by $\Gamma_S = \{(x_1, x_2) : x_2 = \hat{x}_2, x_1 < \hat{x}_1\}$ and the energy is replaced by
\[
\mathcal{E}_S(\bm{u}) = \sum_{\ell \in \Lambda} V\big(D(S_0u_0 + u)\big).
\]
For this energy, we have that $\<\delta \mathcal{E}_S(S\bm{u}^\infty), \bm{v}\> = 0$ since
\begin{align*}
&\<\delta \mathcal{E}_S(S\bm{u}^\infty), \bm{v}\> \\
&=~ \sum_{\ell \in \Lambda} \sum_{\triple\in\Rc}V_{, \triple}\big(D(S_0u_0 + Su^\infty)\big)\cdot D_{\triple}\bm{v}(\ell) \\
&=~ \sum_{\ell \in \Lambda} \sum_{\tripleR}V_{, \triple}\big(D(S_0u_0 + Su^\infty)\big)\cdot\big(D_{\rho}v_0(\ell) + q_\beta(\ell+\rho) - q_\alpha(\ell)\big) \\
&=~ \sum_{\ell \in \Lambda} \sum_{\tripleR}SRV_{, \triple}\big(D(S_0u_0 + Su^\infty)\big)\cdot\big(D_{\rho}v_0(\ell) +q_\beta(\ell+\rho) - q_\beta(\ell)+ q_\beta(\ell)  - q_\alpha(\ell) \big) \\
&=~ \sum_{\ell \in \Lambda} \sum_{\tripleR}V_{, \triple}\big(RD(S_0u_0 + Su^\infty)\big)\cdot\big(R(D_{\rho}v_0(\ell) +q_\beta(\ell+\rho) - q_\beta(\ell)+ q_\beta(\ell)  - q_\alpha(\ell))\big) \\
&=~ \sum_{\ell \in \Lambda} \sum_{\tripleR}V_{, \triple}\big(e(\ell) + \tilde{D}u^\infty)\big)\cdot\big(RD_{\rho}v_0(\ell) + RD_\rho q_\beta(\ell)  + Rq_\beta(\ell)  - Rq_\alpha(\ell) )\big) \\
&=~ \sum_{\ell \in \Lambda} \sum_{\tripleR}V_{, \triple}\big(e(\ell) + \tilde{D}u^\infty)\big)\cdot\big(RD_{\rho}Sw_0(\ell)  + RD_{\rho}Sr_\beta(\ell)  +  r_\beta(\ell)  - r_\alpha(\ell) )\big) \\
&=~ \sum_{\ell \in \Lambda} \sum_{\tripleR}V_{, \triple}\big(e(\ell) + \tilde{D}u^\infty)\big)\cdot \tilde{D}\bm{w}(\ell)  \\
&=~  \<\delta \mathcal{E}(\bm{u}^\infty), \bm{w}\> = 0,
\end{align*}
where $w_0 = Rv_0$, $r_\gamma = Rq_\gamma$, and $w_\gamma = w_0 + r_\gamma$. As we have only reflected the branch cut, this problem is identical to the previous case in the sense that an estimate for $S\bm{u}^\infty$ now exists in the $x \geq \hat{x}_1$ plane (where there is no branch cut):
\[
\max\{|DSU^{\infty}(\ell)|, |S\bm{p}^\infty(\ell)|\} \lesssim~ |\ell|^{-1} + \|(1+|\ell-k|)^{-1}DS\bm{u}^\infty(k)\|_{\ell^2(\Lambda \cap B_{s_2}(\ell))}.
\]
As $R$ simply represents a translation operation by one Burger's vector, we then also have
\[
\max\{|RDSU^{\infty}(\ell)|, |RS\bm{p}^\infty(\ell)|\} \lesssim~ |\ell|^{-1} + \|(1+|\ell-k|)^{-1}RDS\bm{u}^\infty(k)\|_{\ell^2(\Lambda \cap B_{s_2}(\ell))},
\]
or
\[
\max\{|\tilde{D}U^{\infty}(\ell)|, |\bm{p}^\infty(\ell)|\} \lesssim~ |\ell|^{-1} + \|(1+|\ell-k|)^{-1}\tilde{D}\bm{u}^\infty(k)\|_{\ell^2(\Lambda \cap B_{s_2}(\ell))}.
\]

It is now immediate from the proof of~\cite[Lemma 6.7]{Ehrlacher2013} that
\begin{equation}\label{sub_optimal}
\max\{|\tilde{D}U^{\infty}(\ell)|, |\bm{p}^\infty(\ell)|\} \lesssim~ |\ell|^{-1}.
\end{equation}

\subsection{Proof of Theorem~\ref{decay_thm} }
Thus far, we have carefully set everything up so that we may use the techniques and analysis of~\cite{Ehrlacher2013} to complete the proof of Theorem~\ref{decay_thm}.  The primary idea is to obtain suboptimal estimates as in~\eqref{sub_optimal}, which translates into higher regularity of the residual in Theorem~\ref{lin_prob_thm}, which can then in turn be used to prove higher regularity of $\tilde{D}\bm{u}^\infty$.

We return to equation~\eqref{conv_var} and set
\[
\bm{v} =  \sum_m \begin{pmatrix}
D_\tau G_{00}(k-\ell)      & G_{0\bm{p}}(k-\ell)\\
D_\tau G_{\bm{p}0}(k-\ell) &G_{\bm{p}\bm{p}}(k-\ell)
\end{pmatrix}\hat{e}_m
\]
 to write
\begin{equation}\label{conv_diff2}
\begin{split}
&\begin{pmatrix}
D_\tau U^\infty(\ell) \\
\bm{p}^\infty(\ell)
\end{pmatrix} \\
&=
\sum_{k \in \Lambda} \<\delta^2 V(\bm{0}) D\bm{u}^\infty, D\bm{v}\> +
\sum_{k \in \Lambda} \<\delta^2 V(\bm{0}) \tilde{D}\bm{u}^\infty, \tilde{D}\bm{v}\>  -\sum_{k \in \Lambda} \<\delta^2 V(\bm{0}) \tilde{D}\bm{u}^\infty, \tilde{D}\bm{v}\> \\
&= \sum_{k \in \Lambda} \Big(\sum_{\triple \in \mathcal{R}} \bar{f}_{\triple}(k)\cdot \tilde{D}_{\triple}\bm{v}(k) + \<\bm{g}(k), Dv_0(k)\> + \<\bm{k}(k),\bm{q}(k)\>\Big) \quad \mbox{by Theorem~\ref{lin_prob_thm}} \\
&\quad+  
\sum_{k \in \Lambda} \<\delta^2 V(\bm{0}) D\bm{u}^\infty, D\bm{v}\> - \sum_{k \in \Lambda} \<\delta^2 V(\bm{0}) \tilde{D}\bm{u}^\infty, \tilde{D}\bm{v}\>.
\end{split}
\end{equation}

It follows from Equation~\eqref{sub_optimal} that $|\bar{f}_{\triple}(k)| \lesssim (1+ |k|)^{-2}$ and from the decay estimates on $\mathcal{G}$ that
\begin{equation}\label{next_eq}
\begin{split}
\Big|\sum_{k \in \Lambda} \Big(\sum_{\triple \in \mathcal{R}} \bar{f}_{\triple}(k)\cdot \tilde{D}_{\triple}\bm{v}(k) + \<\bm{g}, Dv_0\> + \<\bm{k},\bm{q}\>\Big)\Big| \lesssim~& \sum_{k \in \Lambda} (1+ |k|)^{-2}(1+|\ell-k|)^{-2} \\\
\lesssim~& |\ell|^{-2}\log|\ell|.
\end{split}
\end{equation}
Next, we assume $\ell_1 < \hat{x}_1$ so that we have $\tilde{D} = D$ for $\ell$ large enough, which means
\begin{equation*}
\begin{split}
\sum_{k \in \Lambda} \<\delta^2 V(\bm{0}) D\bm{u}^\infty, D\bm{v}\> - \sum_{k \in \Lambda} \<\delta^2 V(\bm{0}) \tilde{D}\bm{u}^\infty, \tilde{D}\bm{v}\>
\end{split}
\end{equation*}
is nonzero for \textit{only} those $k \in \Lambda$ within $r_{\rm cut}$ of the branch cut, $\Gamma$.  If we let $U_\Gamma$ (as in~\cite{Ehrlacher2013}) denote the set of all such $k$, we have
\begin{equation}\label{sec_eq}
\begin{split}
&
\sum_{k \in \Lambda} \<\delta^2 V(\bm{0}) D\bm{u}^\infty, D\bm{v}\> - \sum_{k \in \Lambda} \<\delta^2 V(\bm{0}) \tilde{D}\bm{u}^\infty, \tilde{D}\bm{v}\> \\
&=~  \sum_{k \in U_\Gamma} \<\delta^2 V(\bm{0}) D\bm{u}^\infty, D\bm{v}\>- \sum_{k \in U_\Gamma} \<\delta^2 V(\bm{0}) \tilde{D}\bm{u}^\infty, \tilde{D}\bm{v}\>.
\end{split}
\end{equation}
In this case we can simply use the suboptimal bounds directly on $\bm{u}^\infty$ and the decay estimates for $\mathcal{G}$ to get 
\begin{equation}\label{third_eq}
\begin{split}
&
\sum_{k \in \Lambda} \<\delta^2 V(\bm{0}) D\bm{u}^\infty, D\bm{v}\> - \sum_{k \in \Lambda} \<\delta^2 V(\bm{0}) \tilde{D}\bm{u}^\infty, \tilde{D}\bm{v}\> \\
&\lesssim~ \sum_{k \in U_\Gamma} (1+|k|)^{-1}(1 + |\ell-k|)^{-2}.
\end{split}
\end{equation}
As $U_\Gamma$ is simply a strip, this summation is now effectively one dimensional and can be bounded by noting $|\ell - k| \gtrsim |\ell| + |k|$ when $\ell_1 < \hat{x}_1$ and $k \in U_\Gamma$ so
\begin{equation}\label{fourth_eq}
\begin{split}
\sum_{k \in U_\Gamma} (1+|k|)^{-1}(1 + |\ell-k|)^{-2} \lesssim~& \int_{|\ell|}^\infty |\ell - k|^{-1}|k|^{-2}\, dk \\
 \lesssim~& \int_{|\ell|}^\infty \frac{1}{|\ell||k|^{1} + |k|^{2}}\, dk \lesssim~ |\ell|^{-2}\log|\ell|.
\end{split}
\end{equation}

In the case $\ell_1 > \hat{x}_1$, then we may simply perform another reflection argument by placing the branch cut in the left-half plane.  It therefore follows from~\eqref{fourth_eq} and~\eqref{next_eq} that in fact
\begin{equation}\label{optimal}
\max\{|\tilde{D}U^{\infty}(\ell)|, |\bm{p}^\infty(\ell)|\} \lesssim~ |\ell|^{-2}\log|\ell|,
\end{equation}
for \textit{all} large enough $|\ell|$.  This completes the proof of Theorem~\ref{decay_thm}.

As for the case for higher order derivatives (differences) alluded to in Remark~\ref{decay_remark}:
\[
\max\{|\tilde{D}^jU^{\infty}(\ell)|, |D^{j-1}\bm{p}^\infty(\ell)|\} \lesssim~ |\ell|^{-1-j}\log|\ell|,
\]
we simply give a high-level view of the argument as it is nearly identical to that given in~\cite[Proof of Theorem 3.6]{Ehrlacher2013} but with our usual modifications to extend to the multilattice case.  The principal idea is to again write
\begin{equation}\label{conv_diff3}
\begin{split}
\begin{pmatrix}
D_\tau U^\infty(\ell) \\
\bm{p}^\infty(\ell)
\end{pmatrix}
=~&
 \sum_m\< H(k)\begin{pmatrix}
U^\infty(k) \\
\bm{p}^\infty(k)
\end{pmatrix} , \begin{pmatrix}
D_\tau G_{00}(k-\ell)      & G_{0\bm{p}}(k-\ell)\\
D_\tau G_{\bm{p}0}(k-\ell) &G_{\bm{p}\bm{p}}(k-\ell)
\end{pmatrix}\hat{e}_m \>
\end{split}
\end{equation}
and take higher finite differences in succession. At each stage, the estimates on the residual can be improved by taking into account the decay proven in the lower order finite differences, and thus improved estimates can be obtained rigorously via an induction argument.

\appendix

\section{Supplementary Proofs}

\subsection{Equivalence of Cauchy-Born Equations}\label{app:cb_linear}

Here we show that ``standard'' Cauchy--Born model is equivalent to the mixed formulation used in the proof of Lemma~\ref{at_force}.  Namely, we wish to show that if $(U,\bm{p})$ solves
\begin{align}
\nabla \cdot (\mathbb{C} \nabla U) =~& 0 \label{stand1}\\
\partial^2_{\bm{p}\bm{p}} W( {\sf 0}, {\bm 0}) \bm{p}^{\rm lin} =~& -\partial^2_{\bm{p}\mF} W({\sf 0}, {\bm 0}) \nabla U^{\rm lin} \label{stand2} 
\end{align}
then $(U,\bm{p})$ also solves the mixed variational equation
\begin{align*}
&\<\partial_{\mF \mF}W(0, \bm{0})\nabla U, \nabla V \> + \sum_{\nu} \< \partial_{\mF p_\nu}W(0, \bm{0}) \nabla U, q_\nu \> + \sum_{\mu}\< \partial_{ p_\mu \mF}W(0, \bm{0}) p_\mu, \nabla V \> \\
&\qquad+ \sum_{\mu,\nu}\< \partial_{p_\mu p_\nu}W(0, \bm{0}) p_\mu, q_\nu \> = 0, \quad \forall (W,\bm{q}) \in {\rm C}^\infty_0.
\end{align*}
and vice versa. The main connection between these two lies in the identity~\cite{olsonOrtner2016}
\begin{align*}
\mathbb{C} = \partial_{\mF\mF}^2 W(0, \bm{0}) - \partial^2_{\mF \bm{p}} W(0, \bm{0}) [\partial^2_{\bm{p}\bm{p}}W(0, \bm{0})]^{-1} \partial^2_{\bm{p}\mF} W(0, \bm{0}).
\end{align*}
Thus to go from the mixed form to the standard form, we simple take a test function pair with $\nabla V = 0$ which allows us to obtain the equation for $\bm{p}^{\rm lin}$.  We may then choose a test pair with $\bm{q} = 0$, integrate by parts and then use the above identity to obtain $\nabla \cdot (\mathbb{C} \nabla U) = 0$.  In the opposite direction, we simply multiply~\eqref{stand1} by a test function $V$, integrate by parts and substitute~\eqref{stand2}.  We then multiply~\eqref{stand2} by a test function $\bm{q}$ and the results of the preceding to computations to yield the mixed form.

\subsection{Proof of Identity~\eqref{slip_fd_id} }\label{tedious}
Here we establish the identity~\eqref{slip_fd_id}:
\begin{equation*}
\tilde{D}_{\triple} \bm{u}(\ell) = \tilde{D}_\rho u_0(\ell) + \tilde{D}_\rho p_\beta(\ell) + p_\beta(\ell) - p_\alpha(\ell),
\end{equation*}
by considering the cases {\bfseries{(1)}} $\ell \notin \Omega_\Gamma$ and {\bfseries{(2)}} $\ell \in \Omega_\Gamma$.  For the former, we have
\begin{align*}
\tilde{D}_{\triple} \bm{u}(\ell) =~& D_{\triple}\bm{u}(\ell), \qquad \ell \notin \Omega_\Gamma, \\
=~& D_\rho u_0(\ell) + D_\rho p_\beta(\ell) + p_\beta(\ell) - p_\alpha(\ell) \\
=~& \tilde{D}_\rho u_0(\ell) + \tilde{D}_\rho p_\beta(\ell) + p_\beta(\ell) - p_\alpha(\ell), \qquad \mbox{since $\ell \notin \Omega_\gamma$.}
\end{align*}
For the latter case $\ell \in \Omega_\Gamma$, we have
{\footnotesize
\begin{align*}
&\tilde{D}_{\triple} \bm{u}(\ell) =~ RD_{\triple}S\bm{u}(\ell), \qquad \ell \in \Omega_\Gamma, \\
&=~ \begin{cases}
&D_{\triple}S\bm{u}(\ell), \, \ell_2 > \hat{x}_2 \\
&D_{\triple}S\bm{u}(\ell+b_{12}), \, \ell_2 < \hat{x}_2
\end{cases} \\
&=~ \begin{cases}
&(S\bm{u})_\beta(\ell+\rho)-(S\bm{u})_\alpha(\ell), \, \ell_2 > \hat{x}_2 \\
&(S\bm{u})_\beta(\ell+b_{12}+\rho)-(S\bm{u})_\alpha(\ell+b_{12}), \, \ell_2 < \hat{x}_2
\end{cases} \\
&=~ \begin{cases}
&(S\bm{u})_\beta(\ell+\rho)-u_\alpha(\ell), \, \ell_2 > \hat{x}_2 \\
&(S\bm{u})_\beta(\ell+b_{12}+\rho)-u_\alpha(\ell), \, \ell_2 < \hat{x}_2
\end{cases} \\
&=~ \begin{cases}
&u_\beta(\ell+\rho)-u_\alpha(\ell), \, \ell_2 > \hat{x}_2, \ell_2+\rho_2 > \hat{x}_2 \\
&u_\beta(\ell+\rho-b_{12})-u_\alpha(\ell), \, \ell_2 > \hat{x}_2, \ell_2+\rho_2 < \hat{x}_2 \\
&u_\beta(\ell+b_{12}+\rho)-u_\alpha(\ell), \, \ell_2 < \hat{x}_2, \ell_2+\rho_2 > \hat{x}_2 \\
&u_\beta(\ell+\rho)-u_\alpha(\ell), \, \ell_2 < \hat{x}_2, \ell_2+\rho_2 < \hat{x}_2
\end{cases} \\
&=~ \begin{cases}
&u_0(\ell+\rho) -u_0(\ell) + p_\beta(\ell+\rho)-p_\beta(\ell) + p_\beta(\ell) - p_\alpha(\ell), \, \ell_2 > \hat{x}_2, \ell_2+\rho_2 > \hat{x}_2 \\
&u_0(\ell+\rho-b_{12})-u_0(\ell)+p_\beta(\ell+\rho-b_{12})-p_\beta(\ell)+p_\beta(\ell)-p_\alpha(\ell), \, \ell_2 > \hat{x}_2, \ell_2+\rho_2 < \hat{x}_2 \\
&u_0(\ell+b_{12}+\rho)-u_0(\ell) + p_\beta(\ell+b_{12}+\rho) - p_\beta(\ell) + p_\beta(\ell) - p_\alpha(\ell), \, \ell_2 < \hat{x}_2, \ell_2+\rho_2 > \hat{x}_2 \\
&u_0(\ell+\rho)-u_0(\ell) + p_\beta(\ell+\rho) - p_\beta(\ell) + p_\beta(\ell) -p_\alpha(\ell), \, \ell_2 < \hat{x}_2, \ell_2+\rho_2 < \hat{x}_2
\end{cases} \\
&=~ \begin{cases}
&(S\bm{u})_0(\ell+\rho)-(S\bm{u})_0(\ell) + (S\bm{p})_\beta(\ell+\rho)-(S\bm{p})_\beta(\ell) + p_\beta(\ell) - p_\alpha(\ell), \, \ell_2 > \hat{x}_2, \ell_2+\rho_2 > \hat{x}_2 \\
&(S\bm{u})_0(\ell+\rho)-(S\bm{u})_0(\ell)+(S\bm{p})_\beta(\ell+\rho)-(S\bm{p})_\beta(\ell)+p_\beta(\ell)-p_\alpha(\ell), \, \ell_2 > \hat{x}_2, \ell_2+\rho_2 < \hat{x}_2 \\
&(S\bm{u})_0(\ell+b_{12}+\rho)-(S\bm{u})_0(\ell+b_{12}) + (S\bm{p})_\beta(\ell+b_{12}+\rho) - (S\bm{p})_\beta(\ell+b_{12}) + p_\beta(\ell) - p_\alpha(\ell), \, \ell_2 < \hat{x}_2, \ell_2+\rho_2 > \hat{x}_2 \\
&(S\bm{u})_0(\ell+b_{12}+\rho)-(S\bm{u})_0(\ell+b_{12}) + (S\bm{p})_\beta(\ell+b_{12}+\rho) - (S\bm{p})_\beta(\ell+b_{12}) + p_\beta(\ell) - p_\alpha(\ell), \, \ell_2 < \hat{x}_2, \ell_2+\rho_2 < \hat{x}_2
\end{cases} \\
&=~ \begin{cases}
&D_\rho (S\bm{u})_0(\ell)+D_\rho (S\bm{p})_\beta(\ell) + p_\beta(\ell) - p_\alpha(\ell), \, \ell_2 > \hat{x}_2, \\
&D_\rho (S\bm{u})_0(\ell+b_{12})+D_\rho (S\bm{p})_\beta(\ell+b_{12}) + p_\beta(\ell) - p_\alpha(\ell), \, \ell_2 < \hat{x}_2,
\end{cases} \\
&=~ RD_\rho (S\bm{u})_0(\ell)+RD_\rho (S\bm{p})_\beta(\ell) + p_\beta(\ell) - p_\alpha(\ell) = \tilde{D}_\rho u_0(\ell) + \tilde{D}_\rho p_\beta(\ell) + p_\beta(\ell) - p_\alpha(\ell).
\end{align*}
}
Analogously, by replacing $S$ with $S_0$ in the derivation above, it is straightforward to see that
\[
\tilde{D}_{\triple} \bm{u}^0(\ell) =~ \tilde{D}_\rho U^0(\ell) + \tilde{D}_\rho p^0_\beta(\ell) + p^0_\beta(\ell) - p^0_\alpha(\ell).
\]

\bibliographystyle{plainurl}	
\bibliography{myrefs}

\begin{thebibliography}{10}

\bibitem{born1954}
M.~Born and K.~Huang.
\newblock {\em Dynamical Theory of Crystal Lattices}.
\newblock Clarendon Press, first edition, 1954.

\bibitem{2017-bcscrew}
J.~Braun, M.~Buze, and C.~Ortner.
\newblock The effect of crystal symmetries on the locality of screw dislocation
  cores.
\newblock {\em SIAM J. Math. Anal.}, 51, 2019.
\newblock \href {https://doi.org/10.1137/17M1157520}
  {\path{doi:10.1137/17M1157520}}.

\bibitem{buze2021numerical}
Maciej Buze and James~R Kermode.
\newblock Numerical-continuation-enhanced flexible boundary condition scheme
  applied to mode-i and mode-iii fracture.
\newblock {\em Physical Review E}, 103(3):033002, 2021.

\bibitem{cauchy}
A.L. Cauchy.
\newblock De la pression ou la tension dans un systeme de points materiels.
\newblock In {\em Exercices de Mathematiques}. 1828.

\bibitem{chang1966edge}
R.~Chang and LJ~Graham.
\newblock Edge dislocation core structure and the peierls barrier in
  body-centered cubic iron.
\newblock {\em Physica Status Solidi (b)}, 18(1):99--103, 1966.

\bibitem{chen17}
H.~Chen and C.~Ortner.
\newblock {QM/MM} methods for crystalline defects. {Part 2: C}onsistent energy
  and force-mixing.
\newblock {\em Multiscale Model. Simul.}, 15:184--214, 2017.

\bibitem{Daw1984a}
M.~S. Daw and M.~I. Baskes.
\newblock Embedded-atom method: Derivation and application to impurities,
  surfaces, and other defects in metals.
\newblock {\em Phys.Rev.B}, 29(12):6443 -- 6453, 1984.

\bibitem{weinan2007cauchy}
W.~E and P.~Ming.
\newblock {C}auchy--{B}orn rule and the stability of crystalline solids: static
  problems.
\newblock {\em Archive for Rational Mechanics and Analysis}, 183(2):241--297,
  2007.

\bibitem{Ehrlacher2013}
V.~Ehrlacher, C.~Ortner, and A.~V. Shapeev.
\newblock Analysis of boundary conditions for crystal defect atomistic
  simulations.
\newblock {\em Archive for Rational Mechanics and Analysis}, 222(3):1217--1268,
  2016.
\newblock \href {https://doi.org/10.1007/s00205-016-1019-6}
  {\path{doi:10.1007/s00205-016-1019-6}}.

\bibitem{hirth1982theory}
J.P. Hirth and J.~Lothe.
\newblock {\em Theory of Dislocations}.
\newblock Krieger Publishing Company, 1982.

\bibitem{2021-qmmm3}
Yangshuai~Wang Huajie~Chen, Christoph~Ortner.
\newblock Qm/mm methods for crystalline defects. part 3: Machine-learned
  interatomic potentials.
\newblock {\em ArXiv e-prints}, 2106.14559, 2021.
\newblock URL: \url{https://arxiv.org/abs/2106.14559}.

\bibitem{2014-dislift}
T.~Hudson and C.~Ortner.
\newblock Analysis of stable screw dislocation configurations in an anti-plane
  lattice model.
\newblock {\em SIAM J. Math. Anal.}, 41:291--320, 2015.

\bibitem{koten_2011}
B.~Van Koten and M.~Luskin.
\newblock Analysis of energy-based blended quasi-continuum approximations.
\newblock {\em SIAM J. Numer. Anal.}, 49(5):2182--2209, 2011.

\bibitem{koten2013}
B.~Van Koten and C.~Ortner.
\newblock Symmetries of 2-lattices and second order accuracy of the
  {C}auchy--{B}orn model.
\newblock {\em SIAM Multiscale Modelling and Simulation}, 11:615--634, 2013.
\newblock \href {https://doi.org/http://dx.doi.org/10.1137/120870220}
  {\path{doi:http://dx.doi.org/10.1137/120870220}}.

\bibitem{nandedkar1987atomic}
A.S. Nandedkar and J.~Narayan.
\newblock Atomic structure of dislocations and dipoles in silicon.
\newblock {\em Philosophical Magazine A}, 56(5):625--639, 1987.

\bibitem{Nandedkar1990}
AS~Nandedkar and J~Narayan.
\newblock Atomic structure of dislocations in silicon, germanium and diamond.
\newblock {\em Philosophical Magazine A}, 61(6):873--891, 1990.

\bibitem{multiBlendingArxiv}
D.~Olson, X.~Li, C.~Ortner, and B.~Van~Koten.
\newblock Force-based atomistic/continuum blending for multilattices.
\newblock {\em Numerische Mathematik}, 140(3):703--754, 2018.

\bibitem{olsonOrtner2016}
Derek Olson and Christoph Ortner.
\newblock Regularity and locality of point defects in multilattices.
\newblock {\em Applied Mathematics Research eXpress}, 2017(2):297--337, 2017.

\bibitem{stillinger1985}
F.~Stillinger and T.~Weber.
\newblock Computer simulation of local order in condensed phases of silicon.
\newblock {\em Physical review B}, 31(8):5262, 1985.

\bibitem{tadmor2011potential}
Ellad~B Tadmor, Ryan~S Elliott, James~P Sethna, Ronald~E Miller, and Chandler~A
  Becker.
\newblock The potential of atomistic simulations and the knowledgebase of
  interatomic models.
\newblock {\em Jom}, 63(7):17, 2011.

\bibitem{volterra}
V.~Volterra.
\newblock Sur l'\'equilibre des corps \'elastiques multiplement connexes.
\newblock {\em Annales scientifiques de l'\'Ecole Normale Sup\'erieure}, 3e
  s{\'e}rie, 24:401--517, 1907.
\newblock \href {https://doi.org/10.24033/asens.583}
  {\path{doi:10.24033/asens.583}}.

\bibitem{wallace1998}
D.~Wallace.
\newblock {\em Thermodynamics of Crystals}.
\newblock Dover, 1998.

\end{thebibliography}

\end{document}